\documentclass{article}
\usepackage[utf8]{luainputenc}
\usepackage{cite}
\usepackage{color}
\usepackage{amscd}
\usepackage{amsmath}
\usepackage{amssymb}
\usepackage{esint}
\usepackage[unicode=true,
bookmarks=false,
breaklinks=false,pdfborder={0 0 1},backref=section,colorlinks=false]
{hyperref}

\makeatletter

\usepackage{amsfonts}
\usepackage{color}
\usepackage{graphicx}
\usepackage{cite}
\providecommand{\U}[1]{\protect\rule{.1in}{.1in}}

\newtheorem{theorem}{Theorem}
\newtheorem{acknowledgement}[theorem]{Acknowledgement}

\newtheorem{corollary}[theorem]{Corollary}

\newtheorem{definition}[theorem]{Definition}
\newtheorem{example}[theorem]{Example}

\newtheorem{fact}[theorem]{Fact}
\newtheorem{lemma}[theorem]{Lemma}

\newtheorem{notation}[theorem]{Notation}

\newtheorem{proposition}[theorem]{Proposition}
\newtheorem{remark}[theorem]{Remark}

\newtheorem{convention}[theorem]{Convention}

\newenvironment{proof}[1][Proof]{\noindent\textbf{#1.} }{\ \rule{0.5em}{0.5em}}

\newcommand{\Ep}{E_1}

\makeatother

\begin{document}
 
	\title{Integration by parts of some non-adapted vector field from Malliavin's lifting approach}

	\author{Zhehua Li}

	\date{\today }
	
	\maketitle
	
	\numberwithin{theorem}{section} \numberwithin{equation}{section}
	
	\begin{abstract}
	 In this paper we propose a lift of vector field $X$ on a Riemannian manifold $M$ to a vector field $\tilde{X}$ on the curved Cameron-Martin space $H\left(
			M\right)$ named orthogonal lift. The construction of this lift is based on a least square spirit with respect to a metric on $H(M)$ reflecting the damping effect of Ricci curvature. Its stochastic extension gives rise to a non-adapted Cameron-Martin vector field on $W_o(M)$. In particular, if $M=\mathbb{R}^d$ with Euclidean metric, then the damp disappears and the lift reduces to the well-known Malliavin's lift. We establish an integration by parts formula for these first order differential operators.
	
	\end{abstract}
	\tableofcontents{}
	
	{}

\section{Introduction\label{cha.1}}
\subsection{Differential Structure on Path Spaces}\label{sec1}
Throughout this paper, we fix $\left(M^{d},g,\nabla ,o\right)$ to be a pointed complete Riemannian manifold of dimension $d$ with Riemannian metric $g$, Levi-Civita covariant derivative $\left(\nabla\right)$, and base point $o\in M$.
We further let 
\[
W_o\left(M\right):=\left\{ \sigma\in C\left(\left[0,1\right]\mapsto M\right)\mid\sigma\left(0\right)=o\right\} 
\]
be the \textbf{Wiener space} on $M$ and let $\nu$ be the \textbf{Wiener measure} on
$W_o\left(M\right)$---i.e. the law of $M$--valued Brownian motion
which starts at $o\in M.$ In order to highlight the effect of curvature in our paper we reserve the symbol $\left(W_0(\mathbb{R}^d),\mu\right)$ for the Wiener space and Wiener measure on $W_0(\mathbb{R}^d)$ and refer to this pair as the classical Wiener space. In contrast, $\left(W_o(M),\nu\right)$ is usually refereed to as curved Wiener space.

Differential calculus on $W_o(M)$ which is compatible with $\nu$ has been extensively explored and has been the main tool of modern stochastic analysis. The first question in this direction is to specify a differential structure (tangent space $\mathcal{X}$ of the path space) that is compatible with Wiener measure $\nu$, i.e. for any vector field (first order differential operator) $Y\in \mathcal{X}$, we can find an \textquotedblleft integral curve\textquotedblright or \textquotedblleft flow\textquotedblright      $\phi_t$ at least in probability, such that 
\[Y(\phi_t)=\frac{d}{dt}\phi_t\text{ or }Yf(\phi_t)=\frac{d}{dt}f(\phi_t)\text{ for some }\nu-\text{measurable function }f.\]
A minimum requirement to achieve the above result is the well-definedness of $f(\phi_t)$, i.e. the law of $\phi_t:W_o(M)\to W_o(M)$ should be equivalent to $\nu$. Cameron and Martin \cite{CameronandMartin1944} first proposed a differential structure named Cameron-Martin space which is further developed as the most natural tangent space on abstract Wiener space, see Theorem \ref{thm4.3}.
\begin{definition}[Cameron-Martin space]
	Let
	\begin{equation*}
	H\left(\mathbb{R}^d\right):=\left\{ \sigma\in C\left(\left[0,1\right]\mapsto \mathbb{R}^d\right):\sigma\left(0\right)=0\text{ , }\sigma\text{ is a.c. and }\int_{0}^{1}\left\vert \sigma^{\prime}\left(s\right)\right\vert ^{2}ds<\infty\right\}
	\end{equation*}be the\textbf{ Cameron-Martin space} on $\mathbb{R}^d$. $\left(\text{Here a.c. means absolutely continuous.}\right)$
\end{definition}
\begin{theorem}[Cameron-Martin]\label{thm4.3}For any $h\in H\left(\mathbb{R}^{d}\right),$
	consider the flow $\phi_{t}^{h}$ generated by $h$, i.e. for any
	$w\in W_0\left(\mathbb{R}^{d}\right)$, $\phi_{t}^{h}\left(w\right)=w+th.$ Notice that $\phi_{t}^{h}$
	is the flow of the vector field $D_{h}:=\frac{\partial}{\partial h}.$
	Then the pull--back measure $\mu^{h}\left(\cdot\right):=\left(\phi_{1}^{h}\right)_{*}\mu\left(\cdot\right)=\mu\left(\cdot-h\right)$
	and Wiener measure $\mu$ are equivalent. \end{theorem}
The map $\phi_{t}^{h}$ is usually called Cameron-Martin shift and the phenomenon
described in Theorem \ref{thm4.3} is called quasi-invariance of $\mu$
under the Cameron-Martin shift. The generalization of Cameron-Martin
Theorem to curved Wiener space came quite a while later in 1990s.
Driver initiated the geometric Cameron-Martin theory in \cite{Driver1992}
and \cite{Driver1994a} where he considered a Cameron-Martin vector field $X^h$ (see Definition \ref{cmv}) in which $h\in \left\{f\in C^{1}\left(\left[0,1\right]\right): f\left(0\right)=0 \right\}\subset H\left(\mathbb{R}^{d}\right).$

\begin{theorem}[Driver] \label{thm4.21}Let $\left(M,g,o,\nabla\right)$ be a compact
	manifold and $h$ be as above, then for any $\sigma\in W_o\left(M\right),$
	there exists a unique flow $\phi_{t}^{h}$ of $X^{h}$, i.e. $\phi_{t}^{h}:W_o\left(M\right)\mapsto W_o\left(M\right)$
	satisfying: 
	\[
	\frac{d}{dt}\phi_{t}^{h}\left(\sigma\right)=X^{h}\left(\phi_{t}^{h}\left(\sigma\right)\right)\text{ with }\phi_{0}^{h}=I
	\]
	and $\nu_{t}^{h}\left(\cdot\right):=\left(\phi_{t}^{h}\right)_{*}\nu$
	is equivalent to $\nu.$ \end{theorem}
The existence of the flow and the quasi-invariance of Wiener measure under this flow were later extended to Cameron-Martin vector field $X^{h}$ with $h\in H\left(\mathbb{R}^{d}\right)$
in \cite{Hsu1995} and \cite{StroockEnchev1995} and then to a geometrically
and stochastically complete Riemannian manifold in \cite{Hsu2002a}
and \cite{HsuOuyang2009}. Meanwhile certain flaws of these Cameron-Martin vector
fields also arise. For example, it has been known that this space of vector fields does not form a Lie Algebra, see \cite{Cruzeiro1996} and \cite{Aida97}, and
 also the It\^{o} map fails to be a diffeomorphism from $W_0(\mathbb{R}^d)$ to $W_o(M)$. Motivated by these issues, Driver introduced more general Cameron-Martin vector field in \cite{Driver1995a}, see also \cite{Cruzeiro1996}, where $h$ admits some randomness. It has been known that if $h$ is certain adapted Brownian semi-martingale, see Definition \ref{def.2.1}, then a quasi-invariant flow can be constructed on $\left(W_0(\mathbb{R}^d),\mu\right)$ and with the help of It\^{o} map, an approximate flow (not a real flow) can be constructed to define $X^h$ on $\left(W_o(M),\nu\right)$. 
In this paper we consider a class of non-adapted Cameron-Martin vector field on $W_o(M)$, see Definition \ref{def Oc}. The reason to study these vector fields is that
 they naturally arise from Malliavin's lifting approach applied to a curved Wiener space where damp is considered. Since Malliavin's lifting approach is the key tool of stochastic analysis in the study of hypo-elliptic differential operators, see \cite{Bell05}, and damping effect naturally appears because of non-trivial curvature, see \cite{Cruzeiro02} and \cite{Elworthy06}, it should be useful to study these non-adapted Cameron-Martin vector fields.
       
\subsection{Riemannian Metrics on $H(M)$ and Lifting Technique}\label{subsec1}

In this section we introduce the Cameron-Martin space on $\left(M,o\right)$ which is a sub-manifold of $W_o(M)$. Its importance are twofold: First, the differential structure on $W_o(M)$, i.e. Cameron-Martin vector field (see Definition \ref{cmv}) can be viewed as a stochastic extension of the differential structure on $H(M)$. Secondly, Riemannian metrics on $H(M)$ give rise to a technique that allows us to lift a vector field from $M$ to $W_o(M)$.    
\begin{definition}[Cameron-Martin space on $\left(M,o\right)$]\label{CM}
	Let
	\begin{equation*}
	H\left(M\right):=\left\{ \sigma\in C\left(\left[0,1\right]\mapsto M\right):\sigma\left(0\right)=o\text{ , }\sigma\text{ is a.c. and }\int_{0}^{1}\left\vert \sigma^{\prime}\left(s\right)\right\vert _g^{2}ds<\infty\right\}
	\end{equation*}be the\textbf{ Cameron-Martin space} on $\left(M,o\right)$. $\left(\text{Here a.c. means absolutely continuous.}\right)$
\end{definition}
\begin{notation}
Let $\Gamma\left(TM\right)$ be differentiable sections of $TM$ and $\Gamma_{\sigma}\left(TM\right)$ be differentiable sections of $TM$ along $\sigma\in H\left(M\right)$.
\end{notation}
The space, $H\left(M\right)$, is an infinite dimensional Hilbert manifold which is a central object in problems related to the calculus of variations on $W_o(M)$. Klingenberg  \cite{Klingenberg77} contains a good exposition of the manifold of paths. In particular, Theorem 1.2.9 in \cite{Klingenberg77} presents the differentiable structure of $H(M)$ in terms of atlases. For our purpose, it suffices to just specify its tangent bundle $TH(M)$ and its Riemannian metrics. In this paper we define two metrics on $H\left(M\right)$. $G^1$-metric seems to be a natural metric to geometers, however a damped metric $\left<\cdot,\cdot\right>_{Ric}$ involving Ricci curvature is more widely seen in the literature of stochastic geometry as a way to represent the damping effect of curvature. 
\begin{definition}[$G^1-$metric]
For any $\sigma\in H\left(M\right)$ and $X,Y\in \Gamma_\sigma^{a.c. }\left(TM\right)$, We define a metric $G^1$ as follows:
\[
\left<X,Y\right>_{G^{1}}=\int_{0}^{1}\left\langle \frac{\nabla X}{ds}\left(s\right),\frac{\nabla Y}{ds}\left(s\right)\right\rangle _{g}ds,
\]
where $\Gamma_\sigma^{a.c.}\left(TM\right)$ is the set of absolutely continuous vector fields along $\sigma$ with finite
energy, i.e. $\int_{0}^{1}\left\langle \frac{\nabla X}{ds}\left(s\right),\frac{\nabla X}{ds}\left(s\right)\right\rangle _{g}ds<\infty$.
\end{definition}

\begin{remark}
To see that $G^1$ is a metric on $H\left(M\right)$, we identify the tangent space $T_{\sigma}H\left(M\right)$ with $\Gamma_\sigma^{a.c.}\left(TM\right)$. To motivate this identification, consider a differentiable one-parameter family of curves $\sigma_t$ in $H\left(M\right)$ such that $\sigma_0=\sigma$. By definition of tangent vector, $\frac{d}{dt}\mid_0\sigma_t\left(s\right)$ should be viewed as a tangent vector at $\sigma$. This is actually the case, for detailed proof, see Theorem 1.3.1 in \cite{Klingenberg77}.  
\end{remark}
\begin{definition}\label{dam}Let $\left\langle \cdot,\cdot\right\rangle _{Ric}$
be the \textbf{damped metric }on $TH\left(M\right)$ defined by 
\begin{equation}
\left\langle X,Y\right\rangle _{Ric}:=\int_{0}^{1}\left\langle \left[\frac{\nabla}{ds}+\frac{1}{2}Ric\right]X\left(s\right),\left[\frac{\nabla}{ds}+\frac{1}{2}Ric\right]Y\left(s\right)\right\rangle_g ds
\end{equation}
for all $X,Y\in\Gamma_{\sigma}^{a.c.}\left(TM\right)=T_{\sigma}H\left(M\right)$
and $\sigma\in H\left(M\right)$. Here $Ric$ is the Ricci tensor, see Notation \ref{Geo}. \end{definition}
\begin{remark}
A damped metric or connection naturally appears when a manifold is involved in order to illustrate the damping effect that comes from the curvature. Other than the literature mentioned at the end of Section \ref{sec1}, in another paper of the author [arxiv], we find an interesting phenomenon that if one discretizes $\left(H(M),G^1\right)$ by considering a class of piecewise geodesic space $H_{\mathcal{P}}(M)$ adapted to a partition $\mathcal{P}$ of time with a metric $G_{\mathcal{P}}^1$ which is the Riemann sum approximation to the $G^1$ metric, then  the orthogonal lift with respect to $G_{\mathcal{P}}^1$-metric as well as its adjoint converges to those of the orthogonal lift with respect to the damped metric on $W_o(M)$.   
\end{remark}
In the category of differential geometry lifting approach is fairly concise to state. Given two differentiable manifold  $N, M$ and a submersion $F:N\to M$, for any differentiable function $f$ on $M$, its lift $\tilde{f}$ with respect to $F$ is simply defined to be $f\circ F$ and for any $X\in \Gamma(TM)$, $\tilde{X}\in \Gamma(TN)$ is called a \textbf{lift} of $X$ iff $F_*\tilde{X}=X$. Since $F$ is a submersion, the existence of $\tilde{X}$ is trivial but one should not expect uniqueness. Based on simple definition one can obtain 
\[\tilde{X}\tilde{f}=\widetilde{Xf}.\]
On $\left(W_o(M),\nu\right)$ one would pursue the above formula in an average sense, i.e.
\begin{equation}
\mathbb{E}_\nu\left[\tilde{X}\tilde{f}\right]=\mathbb{E}_\nu\left[\widetilde{Xf}\right]\text{  }\forall f\in C_b^1(M).\label{1.1}
\end{equation}
In this paper we found a lift named orthogonal lift on $W_o(M)$ with respect to the end point evaluation map $\Ep$ in the following way: first we establish a unique lift of a vector field $X\in \Gamma(TM)$ to $\Gamma(TH(M))$ by requiring it to have minimum norm induced from the damped metric defined in Definition \ref{dam}, then a Cameron-Martin vector field is obtained by stochastic extension.

Since the orthogonal lift $\tilde{X}$ is a non-adapted vector field on the curved Wiener space, it is not clear whether $\tilde{X}$ is in the domain of the divergence operator on $W_o(M)$ or not. To the author's knowledge, even the characterization of the domain of the divergence operator on $W_0(\mathbb{R}^d)$ is not quite satisfactory. Therefore in this paper we adopt a weaker notion of differentiability than the well-known $H-$derivative. However it will be shown that it is enough to derive an integration by parts formula. 

\subsection{Main Theorems\label{sec.1.4}}
In this section we state the main results of this paper while avoiding many technical details.
\begin{lemma}
	Let $\Ep:W_o(M)\to M$ be the \textbf{End point evaluation map}, i.e. $\forall \sigma\in W_o(M)$, $\Ep(\sigma)=\sigma(1)$, then $\Ep\mid_{H(M)}$ is a submersion.
\end{lemma}
\begin{proof}
	Since $M$ is complete, for any $x\in M$, there exists a geodesic $\sigma\in H(M)$ such that $\sigma(0)=o$ and $\sigma(1)=x$. So $\Ep\mid_{H(M)}$ is surjective. Then for any $\sigma\in H(M)$ and $v\in T_{\sigma(1)}M$, set $h(s)=su^{-1}(1)v, 0\leq s\leq 1$ and it is trivial to check $X^h(\sigma,\cdot)\in T_\sigma(H(M))$ and $\left(\Ep\mid_{H(M)}\right)_{\sigma,*}(X^h)=Y$. So $\left(\Ep\mid_{H(M)}\right)_{\sigma,*}$ is surjective and thus $\Ep\mid_{H(M)}$ is a submersion. $($I am sorry for using some notations that have not been set up. Here $u(\sigma,\cdot)$ is the parallel translation along $\sigma$, see Definition \ref{para} and $X^h$ is defined in Notation \ref{pvf} $)$. 
\end{proof}
\begin{theorem}[Orthogonal Lift on $H(M)$]\label{thm1}
	If $M$ has non-positive and bounded sectional curvature, then for any $X\in \Gamma(TM)$, there is a unique $\tilde{X}\in \Gamma(TH(M))$ such that for any $\sigma\in H(M)$, 
	\begin{equation}
	\left\Vert\tilde{X}(\sigma)\right\Vert_{Ric}=\inf\left\{\left\Vert Y(\sigma)\right\Vert_{Ric}: {\Ep}_{*}Y=X\right\}.
	\end{equation}
	where $\left\Vert\cdot\right\Vert_{Ric}$ is the norm on $T_\sigma H(M)$ induced by the damped metric in Definition \ref{dam}. 
\end{theorem}
If we further consider its stochastic extension to $W_o(M)$, we get a non-adapted Cameron-Martin vector field (Still denoted by $\tilde{X}$), then we can prove:
\begin{theorem}\label{thm2}
	Denote by $\mathcal{D}(\tilde{X})$ the domain of $\tilde{X}$ which is dense on $L^2\left(W_o(M),\nu\right)$, then for any $f,g\in \mathcal{D}(\tilde{X})$, we have
	\begin{equation}
	\mathbb{E}_\nu\left[\tilde{X}f\cdot g\right]=\mathbb{E}_\nu\left[f\cdot\tilde{X}^{\dagger}g\right]
	\end{equation}
	where $\tilde{X}^{\dagger}$ is a densely defined operator on $L^2\left(W_o(M),\nu\right)$ explicitly given in Lemma \ref{lem.5.1}.
\end{theorem}  

\subsection{Structure of the Paper\label{sec.1.5}}
For the guidance to the reader, we give a brief summary of the contents of this paper. 

In Section \ref{cha.2} we set up some notations and preliminaries
in probability and geometry. In particular we present the stochastic parallel translation which leads to the stochastic extension of $\tilde{X}$ mentioned in Theorem \ref{thm1} to $W_o(M)$. 

In Section \ref{cha.3} we prove Theorem \ref{thm1} in a constructive way and derive its stochastic extension accordingly.

In Section \ref{cha.4} we first explore the possibility of fitting $\tilde{X}$ into existent theory by summarizing some classical results in differential calculus on $W_o(M)$. Some difficulties are mentioned in this direction. Then we set up a differential calculus for $\tilde{X}$ on $W_o(M)$ and derive an integration by parts formula for it. In the last of this section we explore the divergence term of the adjoint of $\tilde{X}$ under the condition that the curvature tensor is parallel.

\begin{acknowledgement}
I want to thank my advisor Bruce Driver for introducing to me Malliavin's lifting approach, especially its non-adapted nature, in contrast to Bismut's adapted lifting approach, both are powerful tools in Stochastic analysis.
\end{acknowledgement}

\section{Preliminaries in Geometry and Probability\label{cha.2}}
For the remainder of this paper, let $u_0:\mathbb{R}^d\to T_oM$ be a fixed linear isometry which we add to the standard setup $\left(M,g,o,u_0, \nabla\right)$. We use $u_0$ to identify $T_oM$ with $\mathbb{R}^d$. Suggested references for this section are Section 2 of \cite{Hsu01} and Sections 2, 3 of \cite{Driver1992}. Some other references are \cite{Andersson1999},  \cite{Elworthy82}, \cite{Cruzeiro1996} and \cite{Driver1995b} to name just a few.

\begin{definition}[Orthonormal Frame Bundle $\left(\mathcal{O}\left(M\right),\pi\right)$]\label{example2}

For any $x\in M$, denote by $\mathcal{O}\left(M\right)_{x}$ the
space of orthonormal frames on $T_{x}M$, i.e. the space
of linear isometries from $\mathbb{R}^{d}$ to $T_{x}M$. Denote $\mathcal{O}\left(M\right):=\cup_{x\in M}\mathcal{O}\left(M\right)_{x}$
and let $\pi:\mathcal{O}\left(M\right)\to M$ be the $\left(\text{fiber}\right)$ projection
map, i.e. for each $u\in\mathcal{O}\left(M\right)_{x}$, $\pi\left(u\right)=x$.
The pair $\left(\mathcal{O}\left(M\right),\pi\right)$ is the orthonormal
frame bundle over $M$.

\end{definition}

\begin{definition}[Connection on $\mathcal{O}\left(M\right)$]\label{con}
The connection on $\mathcal{O}\left(M\right)$ used in this paper is uniquely specified by the $\mathfrak{so}\left(d\right)$--valued connection form $\omega^\nabla$ on $\mathcal{O}\left(M\right)$ determined by $\nabla$; for any $u\in \mathcal{O}\left(M\right)$ and $X\in T_u\mathcal{O}\left(M\right)$, 
\[\omega_u^\nabla\left(X\right):=u^{-1}\frac{\nabla u\left(s\right)}{ds}\mid_{s=0}\]
where $u\left(\cdot\right)$ is a differentiable curve on $\mathcal{O}\left(M\right)$ such that $u\left(0\right)=u$ and $\frac{du\left(s\right)}{ds}\mid_{s=0}=X$. For any $\xi\in \mathbb{R}^d$, $\frac{\nabla u\left(s\right)}{ds}\mid_{s=0}\xi:=\frac{\nabla u\left(s\right)\xi}{ds}\mid_{s=0}$ is the covariant derivative of $u\left(\cdot\right)\xi$ along $\pi\left(u\left(\cdot\right)\right)$ at $\pi\left(u\right)$.   
\end{definition}
\begin{definition}[Horizontal Bundle $\mathcal{H}$]
	Given a connection form $\omega^\nabla$, the horizontal bundle $\mathcal{H}\subset T\mathcal{O}(M)$ is defined to be the kernel of $\omega^\nabla$.
\end{definition}
\begin{definition}\label{def.hv}
For any $a\in \mathbb{R}^d$, define the horizontal lift $B_a\in \Gamma\left(\mathcal{H}\right)$ in the following way: for any $u\in \mathcal{O}\left(M\right)$, $B_a(u)\in \mathcal{H}_u\subset T_u\mathcal{O}(M)$ is uniquely determined by
\[\omega_u^\nabla\left(B_a\left(u\right)\right)=0\text{  and  }\pi_*\left(B_a\left(u\right)\right)=ua.\]
\end{definition}
\begin{definition}[Horizontal Lift of a Path]\label{para}
For any $\sigma\in H\left(M\right)$, a curve
$u:\left[0,1\right]\to\mathcal{O}\left(M\right)$ is said to be a
horizontal lift of $\sigma$ if $\pi\circ u=\sigma$ and $u^\prime\left(s\right)\in \mathcal{H}_{u(s)}\text{ }\forall s\in [0,1]$.
\end{definition}
\begin{remark}
	In this paper we only consider horizontal lift with fixed start point $u_0\in \pi^{-1}\left(\sigma\left(0\right)\right)$. Under this assumption, given $\sigma\in H(M)$, its horizontal lift $u(\sigma,\cdot)$ is unique.
\end{remark}
We denote $u$ by $\psi\left(\sigma\right)$ and call $\psi$ the \textbf{horizontal lift map}.
\begin{definition}[Development Map]\label{def-d}Given
$w\in H\left(\mathbb{R}^{d}\right)$, the solution to the ordinary
differential equation 
\[
du\left(s\right)=\sum_{i=1}^{d}B_{e_i}\left(u\left(s\right)\right)dw^{i}\left(s\right),u\left(0\right)=u_{0}
\]
is defined to be the \textbf{development} of $w$ and we will denote this map $w\to u$ by $\eta$, i.e. $\eta\left(w\right)=u$. Here $\{e_i\}_{i=1}^d$ is the standard basis of $\mathbb{R}^d$.

\end{definition}
\begin{definition}[Rolling Map]

$\phi=\pi\circ\eta:H\left(\mathbb{R}^{d}\right)\to H\left(M\right)$
is said to be the rolling map to $H\left(M\right)$.
\end{definition}
\begin{definition}[Anti-rolling Map]

Given $\sigma\in H\left(M\right)$ with $u=\psi\left(\sigma\right).$
The anti-rolling of $\sigma$ is a curve $w\in H\left(\mathbb{R}^{d}\right)$
defined by:

\[
w_{t}=\int_{0}^{t}u_{s}^{-1}\sigma_{s}^{\prime}ds
\]
\end{definition}
\begin{remark}
It is not hard to see $w=\phi^{-1}\left(\sigma\right)$ and $u(\sigma,s)u_0^{-1}$ is the parallel translation along $\sigma\in H(M)$.	
\end{remark}
A stochastic version of the maps defined above is needed to specify the differential structure on $\left(W_o(M),\nu\right)$. It also provides tools that allow the transition between classical Wiener space and curved Wiener space. Since the
development maps on the smooth category are defined through ordinary
differential equations, a natural way to introduce probability is to replace
ODEs by (Stratonovich) stochastic differential equations.

First we set up some measure theoretic notations and conventions. Suppose
 $\left(\Omega,\left\{ \mathcal{G}_{s}\right\} ,\mathcal{G},P\right)$
is a filtered measurable space with a finite measure $P$. For any
$\mathcal{G}$---measurable function $f$, we use $P\left(f\right)$
and $\mathbb{E}_{P}\left[f\right]$ (if $P$ is a probability measure)
to denote the integral $\int_{\Omega}fdP$. Given two filtered measurable spaces
$\left(\Omega,\left\{ \mathcal{G}_{s}\right\} ,\mathcal{G},P\right)$
and $\left(\Omega^{\prime},\left\{ \mathcal{G}_{s}^{\prime}\right\} ,\mathcal{G}^{\prime},P^{\prime}\right)$
and a $\mathcal{G}/\mathcal{G}^{\prime}$ measurable map $f:\Omega\to\Omega^{\prime}$,
the law of $f$ under $P$ is the push-forward measure $f_{*}P\left(\cdot\right):=P\left(f^{-1}\left(\cdot\right)\right)$.
We will be mostly interested in the path spaces $W_{o}\left(M\right)$,
$W_{0}\left(\mathbb{R}^{d}\right)$ and $W_{u_{0}}\left(\mathcal{O}\left(M\right)\right)$.

\begin{definition}
	Given a Riemannian manifold $Y$, for any $s\in\left[0,1\right]$ let $\Sigma_{s}:W_{y}\left(Y\right)\to Y$
be the \textbf{coordinate functions} given by $\Sigma_{s}\left(\sigma\right)=\sigma\left(s\right)$.
\end{definition}
We will often view $\Sigma$ as a map from $W_{y}\left(Y\right)\text{ to }W_{y}\left(Y\right)$
in the following way: for any $\sigma\in W_{y}\left(Y\right)$ and
$s\in\left[0,1\right]$, $\Sigma\left(\sigma\right)\left(s\right)=\Sigma_{s}\left(\sigma\right)$.
Let $\mathcal{F}_{s}^{o}$ be the $\sigma-$algebra generated by $\left\{ \Sigma_{\tau}:\tau\leq s\right\} $.
We use $\mathcal{F}_{1}^{o}$ as the raw $\sigma-$algebra and $\left\{ \mathcal{F}_{s}^{o}\right\} _{0\leq s\leq1}$
as the filtration on $W_{y}\left(Y\right).$ The next theorem defines
the Wiener measure $\nu$ on
$\left(W_{y}\left(Y\right),\mathcal{F}_{1}^{o}\right).$

\begin{theorem}[Wiener measure]\label{thm Brown}
Assume $Y$ is a stochastically complete Riemannian manifold, then there exists a unique probability measure $\nu$ on $\left(W_{y}\left(Y\right),\mathcal{F}_{1}^{o}\right)$
which is uniquely determined by its finite dimensional distributions
as follows. For any partition $0=s_{0}<s_{1}<\cdots<s_{n-1}<s_{n}=1$
of $\left[0,1\right]$ and bounded functions $f:Y^{n}\to\mathbb{R};$

\begin{equation}
\nu\left(f\left(\Sigma_{s_{1}},\dots,\Sigma_{s_{n}}\right)\right)=\int_{Y^{n}}f\left(x_{1},\dots,x_{n}\right)\Pi_{i=1}^{n}p_{\Delta s_{i}}\left(x_{i-1},x_{i}\right)dx_{1}\cdots dx_{n}\label{eq:-36}
\end{equation}
where $p_{t}\left(\cdot,\cdot\right)$ is the heat kernel on $Y$ associated with $\frac{1}{2}\Delta_g$,
$\Delta_{i}=s_{i}-s_{i-1}$ for $1\leq i\leq n$.  
\end{theorem} 
\begin{definition}[Brownian motion] A stochastic process $X:\left(\Omega,\mathcal{G}_{s},\left\{ \mathcal{G}\right\} ,P\right)$$ \to\left(W_{y}\left(Y\right),\nu\right)$
is said to be a Brownian motion on $Y$ if the law of $X$
is $\nu$ i.e. $X_{*}P:=P\circ X^{-1}=\nu$. \end{definition}
\begin{remark} From Theorem \ref{thm Brown} it is clear that the law
of the adapted process $\Sigma:W_{y}\left(Y\right)\to W_{y}\left(Y\right)$
is $\nu$ and $\Sigma$ is a Brownian motion. We will call $\Sigma$ the \textbf{canonical Brownian motion} on $Y$. \end{remark}
\begin{remark} Using Theorem \ref{thm Brown}, we can construct Wiener
measure on $W_{0}\left(\mathbb{R}^{d}\right)$,
$W_{o}\left(M\right)$ and $W_{u_{0}}\left(\mathcal{O}\left(M\right)\right)$
respectively. In order to avoid ambiguity from moving between $W_{0}\left(\mathbb{R}^{d}\right)$
and $W_{o}\left(M\right)$, as is mentioned at the beginning of the introduction, we fix the symbol $\mu$
as the Wiener measure on $W_{0}\left(\mathbb{R}^{d}\right)$
and reserve the symbol $\nu$ as the Wiener measure on $W_{o}\left(M\right)$. Meanwhile we reserve $\Sigma$
as the canonical Brownian motion on $M$. \end{remark}
\begin{theorem}[Stochastic Horizontal Lift of Brownian Motion]\label{thm sHL}
If $\Sigma$ is the canonical Brownian motion on $M$, then there exists a unique $(\text{up to  }\nu-\text{equivalence})$ $\tilde{u}\in W_{u_0}\left(\mathcal{O}\left(M\right)\right)$ such that 
\begin{equation}
\pi\left(\tilde{u}_s\right)=\Sigma_s.
\end{equation}
\end{theorem}
\begin{proof}
See Theorem 2.3.5 in \cite{Hsu01}
\end{proof}
\begin{definition}[Stochastic Anti--rolling Map]\label{def sar}
If $\Sigma$ is the canonical Brownian motion on $M$, then the stochastic anti--rolling $\beta$ of $\Sigma$ is defined by,
\begin{equation}
d\beta_s=\tilde{u}^{-1}_s\delta\Sigma_s\text{ , }\beta_0=0.
\end{equation}
\end{definition} 
$\tilde{u}$ and $\beta$ defined above are linked through the (stochastic) development map.
\begin{definition}[Stochastic Development Map]\label{def sd}
Let $\tilde{u}$ and $\beta$ be as defined in Theorem \ref{thm sHL} and Definition $\ref{def sar}$, then $\tilde{u}$ satisfies the following SDE driven by $\beta$,
\[
d\tilde{u}_{s}=\sum_{i=1}^{d}B_{e_i}\left(\tilde{u}_{s}\right)\delta \beta_{s}\text{ , }\tilde{u}\left(0\right)=u_{0},
\]
and $\tilde{u}$ is said to be the development of $\beta$. 
\end{definition} 
\begin{fact}\label{fact 1} 
The following facts are well known, the proofs may be found in the references listed at the beginning of this section, for example, Theorem 3.3 in \cite{Driver1992}. 
\begin{itemize}
\item $\phi$ is a diffeomorphism from $H\left(\mathbb{R}^{d}\right)$ to
$H\left(M\right),$ 
\item $\beta$ is a Brownian motion on $\left(W_o\left(\mathbb{R}^d\right), \mu\right)$.
\end{itemize}
\end{fact}
From now on some notations are fixed for the convenience of consistency.
\begin{notation} \label{not2} For any $\sigma\in H\left(M\right)$, $u_{\left(\cdot\right)}\left(\sigma\right)\in H_{u_0}\left(\mathcal{O}\left(M\right)\right)$ is its horizontal lift and $b_{\left(\cdot\right)}\left(\sigma\right)\in H\left(\mathbb{R}^d\right)$ is its anti-rolling. Recall that $\left\{ \Sigma_s\right\}_{0\leq s\leq 1} $ is fixed to be the canonical Brownian motion on $\left(W_o\left(M\right),\nu \right)$. We also fix $\beta\left(\cdot\right)$ to be the stochastic anti-rolling of $\Sigma$, $($which is a Brownian motion on $\mathbb{R}^d)$ and $\tilde{u}\left(\cdot\right)$ to be the stochastic horizontal lift of $\Sigma$. \end{notation}
\begin{notation}[Geometric Notation]\label{Geo}\text{ } 
\begin{itemize}
\item \textbf{curvature tensor} For any $X,Y,Z\in\Gamma\left(TM\right),$ define
the $(\text{Riemann})$ curvature tensor $R:\Gamma\left(TM\right)\times\Gamma\left(TM\right)\to \Gamma\left(End\left(TM\right)\right)$
to be: 
\[
R\left(X,Y\right)Z =\nabla_{X}\nabla_{Y}Z-\nabla_{Y}\nabla_{X}Z-\nabla_{\left[X,Y\right]}Z.
\]
\item For any $\sigma\in H\left(M\right)$, define $R_{u\left(\sigma,s\right)}\left(\cdot,\cdot\right)\cdot$ to be a map from $\mathbb{R}^{d}\otimes\mathbb{R}^{d}$ to $End\left(\mathbb{R}^{d}\right)$ given by;
\begin{equation}
R_{u\left(\sigma,s\right)}\left(a,b\right)\cdot={u\left(\sigma,s\right)}^{-1}R\left(u\left(\sigma,s\right)a,u\left(\sigma,s\right)b\right)u\left(\sigma,s\right)\text{  }\forall a,b\in \mathbb{R}^{d}.\label{n1}
\end{equation} 
where $R$ is the curvature tensor of $M$.
Similarly we define $R_{\tilde{u}\left(\sigma,s\right)}\left(\cdot,\cdot\right)\cdot$ to be a random map $($up to $\nu$-equivalence$)$ from $\mathbb{R}^{d}\otimes\mathbb{R}^{d}$ to $\mathbb{R}^{d}$ as follows:
\begin{equation}
R_{\tilde{u}\left(\sigma,s\right)}\left(\cdot,\cdot\right)\cdot={\tilde{u}\left(\sigma,s\right)}^{-1}R\left(\tilde{u}\left(\sigma,s\right)\cdot,\tilde{u}\left(\sigma,s\right)\cdot\right)\tilde{u}\left(\sigma,s\right).\label{n2}
\end{equation}
\item $Ric\left(\cdot\right):=\sum_{i=1}^{d}R\left(v_{i},\cdot\right)v_{i}$
is the Ricci curvature tensor on $M.$ Here $\left\{ v_{i}\right\} _{i=1}^{d}$
is an orthonormal basis of proper tangent space. Using $u\left(\sigma,s\right)$ or $\tilde{u}\left(\sigma,s\right)$ to pull back $R$ as in $(\ref{n1})$ and $(\ref{n2})$, we can define $Ric_{u\left(\sigma,s\right)}$ and $Ric_{\tilde{u}\left(\sigma,s\right)}$ to be maps $(\text{random maps})$ from $\mathbb{R}^d$ to $\mathbb{R}^d$.
\end{itemize}
\end{notation}
\begin{convention}
	Since most of our results require a curvature bound, it would be convenient to fix a symbol $N$ for it, i.e. $\left\Vert R\right\Vert\leq N$ when it is viewed as a tensor of order 4. Following this manner, we have $\left\Vert Ric \right\Vert\leq (d-1)N$. A generic constant will be denoted by $C$, it can vary from line to line. Sometimes $C_{(\cdot)}$ or $C(\cdot)$ are used to specify its dependence on some parameters.
\end{convention}
\begin{definition}\label{def.rcf} $f:W_o\left(M\right)\mapsto\mathbb{R}$ is a \textbf{cylinder function} if there exists a partition 
\[
\mathcal{P}:=\left\{ 0<s_{1}<\cdots<s_{n}\leq1\right\} 
\]
of $\left[0,1\right]$ and a function $F:C^m\left(M^{n},\mathbb{R}\right)$
such that
\[
f=F\left(\Sigma_{s_{1}},\Sigma_{s_{2}},\dots,\Sigma_{s_{n}}\right).
\]
We denote this space by $\mathcal{FC}^m$. \end{definition}
\begin{notation}\label{not fspace}
Denote 
\[\mathcal{FC}^1_{b}:=\left\{f:=F\left(\Sigma\right)\in \mathcal{FC}^{1}, F\text{ and all its partial differentials }grad_iF\text{ are bounded}\right\}.\]
\end{notation}
\begin{notation}\label{pvf}
Given a measurable function $h:H(M)\to H\left(\mathbb{R}^d\right)$, denote
\[X^h\left(\sigma,s\right):=u\left(\sigma,s\right)h\left(\sigma,s\right).\]
With this notation, we can express, for any $\sigma\in H(M)$, 
\[T_\sigma H(M)=\left\{X^h\mid h:H(M)\to H\left(\mathbb{R}^d\right)\text{ is measurable.} \right\}\]
\end{notation}
\section{The Orthogonal Lift $\tilde{X}$ of $X$ on $H\left(M\right)$ and Its Stochastic
Extension\label{cha.3}}

\subsection{Damped Metrics and Adjoints\label{sec.4.1}}
\begin{notation}
	For any $r,s\in \mathbb{N}$, the $(r,s)$-tensor bundle on $M$ is denoted by $T^{r,s}M$.
\end{notation}
Given $\Lambda\in \Gamma(T^{1,1}M)$, we can define a damped metric on $H(M)$ by replacing $Ric$ with $\Lambda$ in Definition \ref{dam}. Furthermore, for any $\sigma\in H(M)$, using parallel translation $u(\sigma,\cdot)$, one can obtain an isometry between $(T_\sigma H(M),\left<\right>_{\Lambda})$ and $(H(\mathbb{R}^d),\left<\right>_{\alpha})$, where $\alpha(\cdot)=u(\cdot)^{-1}\circ \Lambda\circ u(\cdot)\in C([0,1], \operatorname*{End}\left(\mathbb{R}^{d}\right))$. So in order to prove Theorem \ref{thm1}, there is no more difficulty in considering the following more general metric on $H(\mathbb{R}^d)$.
\begin{definition}[$\alpha$--inner product]\label{def.4.1.1}Let
$\alpha\left(t\right)\in\operatorname*{End}\left(\mathbb{R}^{d}\right)$ be a continuously varying matrix valued function. For $h,k\in H\left(\mathbb{R}^{d}\right)$
let 
\[
\left\langle h,k\right\rangle _{\alpha}:=\int_{0}^{1}\left(\frac{d}{dt}h\left(t\right)+\alpha\left(t\right)h\left(t\right)\right)\cdot\left(\frac{d}{dt}k\left(t\right)+\alpha\left(t\right)k\left(t\right)\right)dt.
\]

\end{definition}

\begin{remark} We denote the norm induced by $\alpha$--inner product
by $\left\Vert \cdot\right\Vert _{\alpha},$ differentiating from the
notation $\left\Vert \cdot\right\Vert _{H\left(\mathbb{R}^{d}\right)}$
for the norm induced by the $H^{1}$-- inner product: $\left\langle h,k\right\rangle _{H^{1}}=\int_{0}^{1}h^{\prime}\left(s\right)\cdot k^{\prime}\left(s\right)ds.$
\end{remark}
For the moment, let $\Ep:H\left(\mathbb{R}^d\right)\to \mathbb{R}^d$ be the end point evaluation map in the case where $M=\mathbb{R}^d$. Let ${\Ep}^*:\mathbb{R}^d\to H\left(\mathbb{R}^d\right)$ be the adjoint of $\Ep$ with respect to the $\alpha$--inner product, i.e. for any $a\in \mathbb{R}^d$ and $h\in H\left(\mathbb{R}^d\right)$, 
\[\left<\Ep h,a\right>_{\mathbb{R}^d}=\left< h,\left({\Ep}^*\right) a\right>_{\alpha}.\]The next theorem computes $\Ep^*$ which is crucial in constructing the orthogonal lift in Subsection \ref{sec4.2}.
\begin{theorem} \label{th.4.1.3}Let $a\in\mathbb{R}^{d}$ and $\alpha\left(t\right)$ be as in Definition \ref{def.4.1.1}, then
$\Ep^{\ast}a\in H\left(\mathbb{R}^{d}\right)$ is given by 
\begin{equation}
\left(\Ep^{\ast}a\right)\left(t\right)=\left(S\left(t\right)\int_{0}^{t}\left[S\left(s\right)^{\ast}S\left(s\right)\right]^{-1}S\left(1\right)^{\ast}ds\right)a.\label{equ.3.1}
\end{equation}
where $S\left(t\right)\in\operatorname*{Aut}\left(\mathbb{R}^{d}\right)$
solves 
\[
\frac{d}{dt}S\left(t\right)+\alpha\left(t\right)S\left(t\right)=0\text{ with }S\left(0\right)=I,
\]
$S(t)^*$ is the conjugate transpose of $S(t)$.

\end{theorem}

\begin{proof} Notice that if $h\left(t\right)=S\left(t\right)w\left(t\right)$
with $w\left(\cdot\right)\in H\left(\mathbb{R}^{d}\right),$ then
\begin{align*}
\left(\frac{d}{dt}+\alpha\left(t\right)\right)h\left(t\right) =\left(\frac{d}{dt}+\alpha\left(t\right)\right)\left[S\left(t\right)w\left(t\right)\right] =\left[\left(\frac{d}{dt}+\alpha\left(t\right)\right)S\left(t\right)\right]w\left(t\right)+S\left(t\right)\dot{w}\left(t\right)=S\left(t\right)\dot{w}\left(t\right).
\end{align*}
And in particular, 
\[
\left\langle Sv,Sw\right\rangle _{\alpha}=\int_{0}^{1}S\left(t\right)\dot{v}\left(t\right)\cdot S\left(t\right)\dot{w}\left(t\right)dt.
\]
Using Lemma \ref{Lem1} we know $S\left(t\right)\in\operatorname*{Aut}\left(\mathbb{R}^{d}\right)$. Given $a\in\mathbb{R}^{d},$ let $w\left(t\right)=\Ep^{\ast}a$ and
define $v\left(t\right):=S\left(t\right)^{-1}w\left(t\right)$ so
that $\Ep^{\ast}a=S\left(t\right)v\left(t\right)$. Then by the definition
of the adjoint we find, 
\begin{align*}
\int_{0}^{1}S\left(t\right)\dot{v}\left(t\right)\cdot S\left(t\right)\dot{w}\left(t\right)dt & =\left\langle Sv,Sw\right\rangle _{\alpha}=\left\langle \Ep^{\ast}a,Sw\right\rangle _{\alpha}=a\cdot\Ep\left(Sw\right)\\
 & =a\cdot S\left(1\right)w\left(1\right)=\int_{0}^{1}S\left(1\right)^{\ast}a\cdot\dot{w}\left(t\right)dt
\end{align*}
As $w\in H\left(\mathbb{R}^{d}\right)$ is arbitrary we may conclude
that 
\[
S\left(t\right)^{\ast}S\left(t\right)\dot{v}\left(t\right)=S\left(1\right)^{\ast}a\implies v\left(t\right)=\int_{0}^{t}\left[S\left(s\right)^{\ast}S\left(s\right)\right]^{-1}S\left(1\right)^{\ast}ads
\]
which proves (\ref{equ.3.1}). \end{proof}

\begin{theorem} \label{the.3.4}If $a\in\mathbb{R}^{d},$ then $h\left(\cdot\right)\in H\left(\mathbb{R}^{d}\right)$
defined by 
\begin{equation}
h\left(t\right):=S\left(t\right)\left(\int_{0}^{t}\left[S\left(s\right)^{\ast}S\left(s\right)\right]^{-1}ds\right)\left(\int_{0}^{1}\left[S\left(s\right)^{\ast}S\left(s\right)\right]^{-1}ds\right)^{-1}S\left(1\right)^{-1}a,\label{equ.3.2}
\end{equation}
is the minimal length element of $H\left(\mathbb{R}^{d}\right)$ such
that $\Ep h=a$, i.e. 
\[
\left\Vert h\right\Vert _{\alpha}=\inf\left\{ \left\Vert k\right\Vert _{\alpha}\mid k\left(\cdot\right)\in H\left(\mathbb{R}^{d}\right),\text{ }\Ep k=a\right\} .
\]

\end{theorem}

\begin{proof} Since $H\left(\mathbb{R}^{d}\right)=\operatorname*{Nul}\left({\Ep}\right)^{\perp}\oplus \operatorname*{Nul}\left({\Ep}\right)$, we have $\Ep h=a\implies \Ep h_k=a$ and $\left\Vert h \right\Vert_\alpha\geq \left\Vert h_k \right\Vert_\alpha$ where $h_k$ is the orthogonal projection of $h$ onto $\operatorname*{Nul}\left({\Ep}\right)^{\perp}$. So we are looking for the element,
$h\in H\left(\mathbb{R}^{d}\right),$ such that $\Ep h=a$ and $h\in\operatorname*{Nul}\left({\Ep}\right)^{\perp}=\operatorname*{Ran}\left({\Ep}^{\ast}\right).$
In other words we should have $h=\Ep^{\ast}v$ for some $v\in\mathbb{R}^{d}.$
Thus, using Eq.(\ref{equ.3.1}), we need to demand that 
\[
a={\Ep}{\Ep}^{\ast}v=\left({\Ep}^{\ast}v\right)\left(1\right)=\left(S\left(1\right)\int_{0}^{1}\left[S\left(s\right)^{\ast}S\left(s\right)\right]^{-1}S\left(1\right)^{\ast}ds\right)v,
\]
i.e. 
\[
v=\left(S\left(1\right)\int_{0}^{1}\left[S\left(s\right)^{\ast}S\left(s\right)\right]^{-1}S\left(1\right)^{\ast}ds\right)^{-1}a.
\]
Here we have used Lemma \ref{Lem1} to show $S(1)$ and $\int_{0}^{1}\left[S\left(s\right)^{\ast}S\left(s\right)\right]^{-1}ds$ are invertible.

It then follows that 
\begin{align*}
& h\left(t\right)={\Ep}^{\ast}\left(S\left(1\right)\int_{0}^{1}\left[S\left(s\right)^{\ast}S\left(s\right)\right]^{-1}S\left(1\right)^{\ast}ds\right)^{-1}a\\
 & =\left(S\left(t\right)\int_{0}^{t}\left[S\left(s\right)^{\ast}S\left(s\right)\right]^{-1}S\left(1\right)^{\ast}ds\right)\left(S\left(1\right)\int_{0}^{1}\left[S\left(s\right)^{\ast}S\left(s\right)\right]^{-1}S\left(1\right)^{\ast}ds\right)^{-1}a
\end{align*}
which is equivalent to Eq.(\ref{equ.3.2}).
\end{proof}

\begin{remark} \label{rem.3.5}The expression in (\ref{equ.3.2})
matches the well known result for damped metric where $\alpha=\frac{1}{2}\operatorname*{Ric}_{u}$.
Further observe that if $\alpha\left(t\right)=0$ (i.e. we are in
the flat case) then $S\left(t\right)=I$ and the above expression
reduces to $h\left(t\right)=ta$ as we know to be the correct result.
\end{remark}
\subsection{The Orthogonal Lift $\tilde{X}$ on $H\left(M\right)$\label{sec4.2}}

In this subsection we construct the orthogonal lift $\tilde{X}\in \Gamma\left(TH\left(M\right)\right)$ of $X\in\Gamma\left(TM\right)$ which is defined to be the minimal length element in $\Gamma\left(TH\left(M\right)\right)$ relative to the damped metric introduced in Definition \ref{dam}.

\begin{definition} \label{def.4.2}For each $\sigma\in H\left(M\right)$, recall that $u_s\left(\sigma\right)$ is the horizontal lift of $\sigma$. Denote by $T_{\left(\cdot\right)}:H\left(M\right)\to End\left(\mathbb{R}^d\right)$ the
solution to the following initial value problem: 
\begin{equation}
\begin{cases}
\frac{d}{ds}T_s+\frac{1}{2}Ric_{u_s}T_s=0\\
T_0=I.
\end{cases}\label{equ.4.3}
\end{equation}

\end{definition}

\begin{lemma} \label{lem.4.3}For all $s\in\left[0,1\right]$, $T_s$
is invertible. Further both $\underset{0\leq s\leq1}{\sup}\left\Vert T_s\right\Vert$ and
$\underset{0\leq s\leq1}{\sup}\left\Vert T_s^{-1}\right\Vert $
are bounded by $e^{\frac{1}{2}\left(d-1\right)N}$, where $\left(d-1\right)N$
is a bound of $\left\Vert\operatorname*{Ric}\right\Vert.$
\end{lemma}

\begin{proof}
	Apply Lemma \ref{Lem1} with $\alpha(s)=-\frac{1}{2}Ric_{u_s}$, one get $T_s$ is invertible $\forall s\in [0,1]$ and $T^{-1}_s$ satisfies the following ODE,
\begin{equation}
\begin{cases}
\frac{d}{ds}U_s=\frac{1}{2}U_s\operatorname*{Ric}\nolimits _{u_s}\\
U_0=I.
\end{cases}\label{equ.4.4}
\end{equation}
The stated bounds now follow
by Gronwall's inequality and the boundedness of curvature tensor. \end{proof}

\begin{definition} \label{def.4.4}Let $\mathbf{K}:\left[0,1\right]\times H\left(M\right)\to End\left(\mathbb{R}^d\right)$ be defined by
\begin{equation}
\mathbf{K}_s:=T_s\left[\int_{0}^{s}T_r^{-1}\left(T_r^{-1}\right)^{\ast}dr\right]T^{\ast}_1.\label{equ.4.6}
\end{equation}
\end{definition}
\begin{lemma} \label{lem.4.6}$\mathbf{K}_1$ is invertible
and $\left\Vert \mathbf{K}_1^{-1}\right\Vert \leq e^{\left(d-1\right)N}$, provided $\left\Vert\operatorname*{Ric}\right\Vert\leq \left(d-1\right)N$.
\end{lemma}

\begin{proof} 
Since
\[
\mathbf{K}_1:=\int_{0}^{1}\left(T_1T_r^{-1}\right)\left(T_1T_r^{-1}\right)^{\ast}dr
\]
is a symmetric positive semi-definite operator such that
\[
\left\langle \mathbf{K}_1v,v\right\rangle =\int_{0}^{1}\left\Vert \left(T_1T_r^{-1}\right)^{\ast}v\right\Vert ^{2}dr\text{  }\forall v\in\mathbb{C}^{d}.
\]
Apply Lemma \ref{lem.4.3} to the expression given;
\begin{align*}
\left\langle \mathbf{K}_1v,v\right\rangle  & \geq\int_{0}^{1}e^{-\left(d-1\right)N}\left\Vert \left(T_r^{-1}\right)^{\ast}v\right\Vert ^{2}dr\geq\int_{0}^{1}e^{-2\left(d-1\right)N}\left\Vert v\right\Vert ^{2}dr=e^{-2\left(d-1\right)N}\left\Vert v\right\Vert ^{2}
\end{align*}
from which it follows that $eig\left(\mathbf{K}_1\right)\subset\lbrack e^{-\left(d-1\right)N},\infty)$ and $\left\Vert\mathbf{K}_1^{-1}\right\Vert=\frac{1}{\min\left\{\lambda:\lambda\in eig(\mathbf{K}_1)\right\}}\leq e^{(d-1)N}$.
\end{proof}

\begin{definition} \label{def.4.7}Let $X\in\Gamma\left(TM\right)$, define two maps $H:H\left(M\right)\to \mathbb{R}^d$ and $J:\left[0,1\right]\times H\left(M\right)\to \mathbb{R}^d$ as follows,
\begin{equation}
H(\sigma)=u_1^{-1}\left(\sigma\right)X\circ {\Ep}\left(\sigma\right)\label{equ.4.8}
\end{equation}
and 
\begin{equation}
J\left(\sigma,s\right):=J_s\left(\sigma\right):=\mathbf{K}_s\left(\sigma\right)\mathbf{K}_1^{-1}\left(\sigma\right)H\left(\sigma\right).\label{equ.4.9}
\end{equation}

\end{definition}

\begin{theorem} \label{the.4.8}Given $X\in\Gamma\left(TM\right)$, the minimal length lift $\tilde{X}$ relative to the damped metric
in Definition \ref{dam} of $X$ to $\Gamma\left(TH\left(M\right)\right)$, is
given by $\tilde{X}=X^{J}$. Further we know that $J_s$ is the solution
to the following ODE: 
\[
J^{\prime}_s=-\frac{1}{2}Ric_{u_s}J_s+\phi_s,\text{ }J_0=0
\]
where $\phi_s=\left(T_1T_s^{-1}\right){}^{\ast}\mathbf{K}_1^{-1}H=\left(T_s^{-1}\right){}^{\ast}\left[\int_{0}^{1}T_r^{-1}\left(T_r^{-1}\right)^{\ast}dr\right]^{-1}T_1^{-1}H.$
\end{theorem}

\begin{proof} Apply Theorem \ref{the.3.4} with $\alpha_s=\frac{1}{2}Ric_{u_s}.$
\end{proof}

The following construction gives rise to a stochastic extension of $\tilde{X}$ to a Cameron-Martin vector field on $W_o(M)$. The definition of Cameron-Martin vector field is given right below. Its properties are further explored in the next section.

Recall from Notation \ref{not2} that $\tilde{u}$ is the stochastic horizontal lift of the canonical Brownian motion $\Sigma$ on $M$. Mimicking the tangent bundle $TH(M)$ of $H(M)$ as expressed in Notation \ref{pvf}, we define a Cameron-Martin vector field $($not necessarily adapted$)$ as follows.
\begin{definition}\label{cmv}
	A \textbf{Cameron-Martin process}, $h$, is an $\mathbb{R}^d$---valued process on $W_o(M)$ such that $s\to h(s)$ is in $H(\mathbb{R}^d)\text{ }\nu-a.s.$ A $TM$-valued process $Y$ on $(W_o(M),\nu)$ is called a \textbf{Cameron-Martin vector field} (denote this space by $\mathcal{X}$) if $\pi(Y_s)=\Sigma_s\text{ }\nu-a.s.$ and there exists a Cameron-Martin process $h(\cdot)$ such that $Y(s)=\tilde{u}_sh_s \forall s\in [0,1]\text{ }\nu-a.s.$ with\[ 
	\left<Y,Y\right>_{\mathcal{X}}:=\mathbb{E}\left[ \left\Vert h\right\Vert^2_{H(\mathbb{R}^d)}\right]<\infty.\]
	We will write $X^h=Y$ to highlight this representation and $X^h$ is called adapted if $h$ is adapted.
\end{definition}

\begin{definition} \label{def.T}Define $\tilde{T}_{\left(\cdot\right)}:\left[0,1\right]\times W_o\left(M\right)\to End\left(\mathbb{R}^d\right)$ to be
 the solution to the following initial value problem: 
\begin{equation}
\begin{cases}
\frac{d}{ds}\tilde{T}_s+\frac{1}{2}Ric_{\tilde{u}_s}\tilde{T}_s=0\\
\tilde{T}_0=I
\end{cases}\label{equ.4.3-1}
\end{equation}
\end{definition}
\begin{definition}\label{def.4.4-1}  Using $\tilde{T}_s$, we define $\mathbf{\tilde{K}}:\left[0,1\right]\times W_o\left(M\right)\to End\left(\mathbb{R}^d\right)$:
	\begin{equation}
	\mathbf{\tilde{K}}_s:=\tilde{T}_s\left[\int_{0}^{s}\tilde{T}_r^{-1}\left(\tilde{T}_r^{-1}\right)^*dr\right]\tilde{T}^{\ast}_1.\label{equ.4.6-1}
	\end{equation}

\end{definition}
\begin{remark} Following the same arguments used in Lemma \ref{lem.4.3}
and \ref{lem.4.6}, one can see the bounds obtained there still hold
for $\tilde{T}$ and $\tilde{\mathbf{K}}\text{ }\nu-a.s$. \end{remark}

\begin{definition} \label{def.4.7-1}For each $X\in\Gamma\left(TM\right)$
define two maps $\tilde{H}:W_o\left(M\right)\to \mathbb{R}^d$ and $\tilde{J}:W_o\left(M\right)\to H\left(\mathbb{R}^d\right)$ by
\begin{equation}
\tilde{H}=\tilde{u}_1^{-1}X\circ \Ep\label{equ.4.8-1}
\end{equation}
and 
\begin{equation}
\tilde{J}_s:=\mathbf{\tilde{K}}_s\mathbf{\tilde{K}}_1^{-1}\tilde{H}\text{ for }s\in \left[0,1\right].\label{equ.4.9-1}
\end{equation}
\end{definition}
\begin{notation} \label{not.4.9}Given a measurable function $h:W_o\left(M\right)\to H\left(\mathbb{R}^d\right)$, let $Z_h:W_o\left(M\right)\to H\left(\mathbb{R}^d\right)$ be the solution to the following initial value problem:
\[
\begin{cases}
{Z_h}^{\prime}\left(s\right)=-\frac{1}{2}Ric_{\tilde{u}_s}Z_h\left(s\right)+h^{\prime}_s\\
Z_h\left(0\right)=0.
\end{cases}
\]

\end{notation}

\begin{definition}[Orthogonal Lift on $W_o(M)$] \label{def Oc}For any $X\in\Gamma\left(TM\right),$
define $\tilde{X}\in \mathcal{X}$ as follows,
\[
\tilde{X}_s=X^{Z_{\Phi}}_s:=\tilde{u}_sZ_{\Phi}\left(s\right)\text{ for }0\leq s \leq 1
\]
where 
\[
\Phi_s=\int_{0}^{s}\left(\tilde{T}_{\tau}^{-1}\right){}^{\ast}\left[\int_{0}^{1}\left(\tilde{T}^{\ast}_r\tilde{T}_r\right){}^{-1}dr\right]^{-1}\tilde{T}_1^{-1}\tilde{H}d\tau.
\]
\end{definition}
In the next section we will specify how this Cameron-Martin vector field act on geometric Wiener functionals.

\section{A Differential Calculus on $W_o(M)$ for $\tilde{X}$}\label{cha.4}
\subsection{Review of Calculus on Wiener Space}

First we review some classical results for adapted Cameron-Martin vector field where $(\text{approximate})$ flows can be constructed.

\begin{definition} [Vector Valued Brownian Semimartingale]
\label{def.2.1} Let $V$ be a finite dimensional vector space. A function
$f:W_o\left(M\right)\times\left[0,1\right]\to V$ is called a Brownian
semimartingale if $f$ has the following representation: 
\[
f\left(s\right)=\int_{0}^{s}Q_{\tau}d\beta_{\tau}+\int_{0}^{s}r_{\tau}d\tau
\]
where $\left(Q_{s},r_{s}\right)$ is a predictable process with values
in $Hom\left(\mathbb{R}^d,V\right)\times V$. We will call $\left(Q_{s},r_{s}\right)$ the
kernels of $f$.\end{definition}

\begin{definition}[$R^q$ and $\mathcal{H}^{q}$ Space] \label{def.2.2}For
each $q\in[1,\infty],$ $f:W_o\left(M\right)\times\left[0,1\right]\to V$ jointly measurable, we define the root mean square norm in $L^{q}\left(W_o\left(M\right),\nu \right)$ to be: 
\[
\left\Vert f\right\Vert _{R^{q}\left(V\right)}\equiv\left\Vert \left(\int_{0}^{1}\left|f\left(\cdot,s\right)\right|_V^{2}ds\right)^{\frac{1}{2}}\right\Vert _{L^{q}\left(W_o\left(M\right),\nu \right)}.
\]
Let $R^q$ be the space of all $f:W_o\left(M\right)\times\left[0,1\right]\to V$ such that $\left\Vert f\right\Vert _{R^{q}}<\infty$ and let $\mathcal{H}^{q}$ be the space of all Brownian semimartingales
such that 
\[
\left\Vert f\right\Vert _{\mathcal{H}^{q}}:=\left\Vert Q^{f}\right\Vert _{R^{q}}+\left\Vert r^{f}\right\Vert _{R^{q}}<\infty.
\]
Here we suppress the range space $V$ as it should be easily determined by the context.
\end{definition}
\begin{definition}[$S^q$ and $\mathcal{B}^{q}$ Space] \label{def.2.5}For
each $q\in[1,\infty]$, $f:W_o\left(M\right)\times\left[0,1\right]\to V$ jointly measurable, we define the supremum norm in $L^{q}\left(W_o\left(M\right),\nu \right)$ to be: 
\[
\left\Vert f\right\Vert _{S^{q}\left(V\right)}\equiv\left\Vert f^* \right\Vert _{L^{q}\left(W_o\left(M\right),\nu \right)}
\]where $f^*$ is the essential supremum of $s\to f\left(\cdot,s\right)$ relative to Lebesgue measure on $\left[0,1\right]$. Let $S^q$ be the space of all $f:W_o\left(M\right)\times\left[0,1\right]\to V$ such that $s\to f(s,\cdot):\left[0,1\right]\to V$ is continuous $\nu-a.s.$ and $\left\Vert f\right\Vert _{S^{q}}<\infty$ and let $\mathcal{B}^{q}$ be the space of all Brownian semimartingales
such that 
\[
\left\Vert f\right\Vert _{\mathcal{B}^{q}}:=\left\Vert Q^{f}\right\Vert _{S^{q}}+\left\Vert r^{f}\right\Vert _{S^{q}}<\infty.
\]
\end{definition}
\begin{lemma}\label{lem s1}
	For any $q\in [1,\infty)$, $f:W_o(M)\times[0,1]\to V$ such that the following norms make sense, we have
	\begin{itemize}
		\item $\left\Vert f\right\Vert _{R^{q}\left(V\right)}\leq \left\Vert f\right\Vert _{S^{q}\left(V\right)}$,
		\item $\left\Vert f\right\Vert _{\mathcal{H}^{q}\left(V\right)}\leq \left\Vert f\right\Vert _{\mathcal{B}^{q}\left(V\right)}$,
		\item $\left\Vert f\right\Vert _{S^{q}\left(V\right)}\leq C_q\left\Vert f\right\Vert _{\mathcal{H}^{q}\left(V\right)}$ for some constant $C_q>0$.
	\end{itemize}
\begin{proof}
	The first two items are trivial, so we will only prove the last item.
	
	Since $f$ has the following representation
	\[f_s=\int_{0}^{s}Q_\tau d\beta_\tau+\int_{0}^{s}r_\tau d\tau,\]
	for any $q\in [1,\infty)$, we have
	\[\left\vert f_s\right\vert^q\leq C_q\left(\left\vert \int_{0}^{s}Q_\tau d\beta_\tau\right\vert^q+\left(\int_{0}^{s}\left\vert r_\tau \right\vert d\tau\right)^q\right)\]
	and thus
	\begin{equation}
	\left\vert f^*\right\vert^q\leq C_q\left(\sup_{0\leq s\leq 1}\left\vert \int_{0}^{s}Q_\tau d\beta_\tau\right\vert^q+\left(\int_{0}^{1}\left\vert r_\tau \right\vert^2d\tau\right)^{\frac{q}{2}}\right).\label{e7}
	\end{equation}
	From Burkholder-Davis-Gundy inequality, we have
	\[\mathbb{E}_\nu\left[\sup_{0\leq s\leq 1}\left\vert \int_{0}^{s}Q_\tau d\beta_\tau\right\vert^q\right]\leq C_q\left\Vert Q\right\Vert^q_{R^q},\]
	then taking expectations on Eq.$(\ref{e7})$ we have
	\[\left\Vert f\right\Vert_{S^q}\leq C_q\left(\left\Vert Q\right\Vert_{R^q}+\left\Vert r\right\Vert_{R^q}\right)= C_q\left\Vert f\right\Vert_{\mathcal{H}^q}.\]
\end{proof}
\end{lemma}
\begin{definition} [Adapted Vector Field]\label{def.2.3}
An adapted vector field on $W_o\left(M\right)$ is an $\mathbb{R}^d$--valued Brownian semimartingale
with predictable kernels $Q_{\cdot}\in \mathfrak{so}\left(d\right)$ and $r_{\cdot}\in L^{2}\left[0,1\right]\text{ }\nu-a.s.$
We denote the space of \textbf{adapted vector fields} by $\mathcal{V}$ and let $\mathcal{V}^{q}$ be $\mathcal{V}\cap\mathcal{H}^{q}$, $q\in [1,\infty]$.
\end{definition}
\begin{notation}
We will use the following notations in this paper: $S^{\infty-}:=\cap_{q\geq 1}{S^{q}}$, $\mathcal{H}^{\infty-}:=\cap_{q\geq 1}{\mathcal{H}^{q}}$, $\mathcal{B}^{\infty-}=\cap_{q\geq 1}{\mathcal{B}^{q}}$ and $\mathcal{V}^{\infty-}=\mathcal{V}\cap{\mathcal{H}^{\infty-}}$.
\end{notation}
\begin{theorem}[Approximate Flow]\label{thm flow}
Let $X^h$ be a Cameron-Martin vector field with $h\in \mathcal{V}^{\infty}\cap\mathcal{B}^{\infty}$, $t\in \mathbb{R}$, then there exists a map $E\left(tX^h\right):W_o\left(M\right)\to W_o\left(M\right)$ such that the law of $E\left(tX^h\right)$ is equivalent to $\nu$ and
\[\frac{d}{dt}\mid_0E\left(tX^h\right)=X^h\text{  in  }\mathcal{B}^{\infty-}.\]
\end{theorem}
\begin{proof}
See Corollary 4.6 in \cite{Driver1999}.
\end{proof}

Using the approximate flow, we will specify a domain of an adapted Cameron-Martin vector field with the aim of setting up an integration by parts formula. A remark about other possible domains are provided after the definition below. 
\begin{definition}\label{pro112}
Let $X^h$ be an adapted Cameron-Martin vector field with $h\in \mathcal{V}^{\infty}\cap\mathcal{B}^{\infty}$ and let  $E\left(tX^h\right):W_o\left(M\right)\to W_o\left(M\right)$ be its approximate flow, then we define the domain of $X^h$ to be 
\[\mathcal{D}(X^h):=\left\{f\in L^{\infty-}\left(W_o(M), \nu\right),\frac{d}{dt}\mid_0f\left(E\left(tX^h\right)\right)\text{ exists in }L^{\infty-}\left(W_o(M), \nu\right)\right\}\subset L^2\left(W_o(M)\right)\]
and define $X^hf:=\frac{d}{dt}\mid_0f\left(E\left(tX^h\right)\right)$.
\end{definition}
\begin{remark}[H-derivative]
	The notion of differentiability in Definition \ref{pro112} is weaker than the one defined using $H$-derivative which allows a Sobolev type analysis on $(W_o(M),\nu)$. However this definition is  sufficient to admit an integration by parts formula, see Lemma \ref{lem.5.1}. Here we provide a very rough picture of how the $H$-derivative is defined.

Given $f\in \mathcal{FC}_b^1$, define the \textbf{gradient operator} $Df\in \mathcal{X}$ as follows,
\begin{equation}
D_sf:=\tilde{u}_s\sum_{i=1}^{n}(s\wedge s_i)\tilde{u}^{-1}_{s_i}grad_iF\label{gradient}
\end{equation}
where $F(\Sigma_{s_1},\cdots,\Sigma_{s_n})$ is a representation of $f$ and $grad_iF$ is the differential of $F$ with respect to the $ith$ variable.

Since $\mathcal{FC}_b^1$ is dense in $L^q(W_o(M),\nu)\text{ }\forall q\geq 1$, we know $D$ is a densely defined operator from $L^q(W_o(M),\nu)$ to $\mathcal{X}$. Furthermore, it is well-known that $D$ is closable in $L^q(W_o(M),\nu)\text{ }\forall q\geq 1$ and the domain of its extension is a Sobolev space of index $(1,q)$ on $W_o(M)$. (We will denote this space by $W^q_1(M)$.) If we treat $D$ as an operator from $L^{\infty-}(W_o(M)):=\cap_{q\geq 1}L^q(W_o(M))$ to $\to \mathcal{X}$ with domain $\mathcal{D}(D):=W_1^{\infty-}(M):=\cap_{q\geq 1}W_1^{q}(M)$, then for any $X\in \mathcal{X}$, we may define $Xf:=\left<Df,X\right>_{G^1}$ and require its domain $\mathcal{D}(X)$ to be $W_1^{\infty-}(M)$. However if $X$ is not adapted, it is not known if $X$ is in the domain of $D^*:\mathcal{X}\to W_1^{\infty-}(M)$--- a fact that easily gives rise to integration by parts. There is also the issue of dependence on initial domain when taking closure for $H$-derivative on curved Wiener space.
\end{remark}

The following example shows some advantages of Definition \ref{pro112}: basically one can show that a class of so called generalized cylinder functions are $X^h$ differentiable by explicit computations. This content is summarized from \cite{Driver1999}. 

\begin{definition}\label{gc} $f:W_o\left(M\right)\mapsto\mathbb{R}$ is called a \textbf{generalized cylinder function} if there exists a partition 
\[
\mathcal{P}:=\left\{ 0<s_{1}<\cdots<s_{n}\leq1\right\} 
\]
of $\left[0,1\right]$ and a bounded function $F\in C^m\left(\mathcal{O}\left(M\right)^{n},\mathbb{R}\right)$
such that: 
\[
f=F\left(\tilde{u}_{s_{1}},\tilde{u}_{s_{2}},\dots,\tilde{u}_{s_{n}}\right)\text{ }\nu-a.s.\]
We further require all the partial differentials of $F$ to be bounded and denote this space by $\mathcal{GFC}^m$. \end{definition}
\begin{notation} Given $k:W_o\left(M\right)\to H\left(\mathbb{R}^{d}\right)$, denote $\int_{0}^{s}R_{\tilde{u}_r}\left(k_{r},\delta\beta_{r}\right)$ by $A_{s}\left<k\right>$ when the integral makes sense, here $\delta$ is the stratonovich differential.  
\end{notation}
\begin{notation}
Suppose $F\in C\left(\mathcal{O}\left(M\right)^n\right)$ and $\mathcal{P}=\left\{0<s_1<\cdots<s_n\leq 1\right\}$ is a partition of $\left[0,1\right]$, set 
\[F\left(u\right)=F\left(u_{s_1},\dots,u_{s_n}\right),\]
then for $A:\left[0,1\right]\to \mathfrak{so}\left(d\right)$ and $h:\left[0,1\right]\to \mathbb{R}^d$, set
\[F^\prime\left(u\right)\left<A+h\right>:=\frac{d}{dt}\mid_0F\left(ue^{tA}\right)+\frac{d}{dt}\mid_0F\left(e^{tB_h}\left(u\right)\right)\]
where $ue^{tA}\left(s\right)=u_se^{tA_s}\in \mathcal{O}\left(M\right)$ and $e^{tB_h}\left(u\right)\left(s\right)=e^{tB_{h_s}}\left(u_s\right)\in \mathcal{O}\left(M\right).$
\end{notation}
\begin{theorem} \label{the.2.9}If $h\in\mathcal{V}^{\infty}\cap \mathcal{B}^{\infty}$, then $ \mathcal{GFC}^1\subset \mathcal{D}\left(X^h\right)$. In more detail, if $f=F\left(\tilde{u}\right)\in \mathcal{GFC}^1$, then
\begin{equation}
X^{h}f=F^{\prime}\left(\tilde{u}\right)\left\langle -A\left\langle h\right\rangle +h\right\rangle\text{ }\nu-a.s. \label{equ.2.4}
\end{equation}
Moreover, if $g\in \mathcal{D}\left(X^h\right)$, then 
\begin{equation}
\mathbb{E}_\nu\left[X^{h}f\cdot g\right]=\mathbb{E}_\nu\left[f\cdot\left(X^{h}\right)^{tr,\nu}g\right]
\end{equation}
where $\left(X^{h}\right)^{tr,\nu}:=-X^{h}+\int_{0}^{1}\left\langle h_{s}^{\prime},d\beta_{s}\right\rangle $.
\end{theorem}
\begin{proof}See Proposition 4.10 in \cite{Driver1999} . \end{proof}

We now construct a class of Cameron-Martin vector field and use it as a basis to expand the orthogonal lift $\tilde{X}$ defined in Definition \ref{def Oc}.
\begin{notation}
Recall from Notation \ref{not.4.9} that $Z_h$ satisfies the following ODE,
\begin{equation}
Z_{h}^{\prime}\left(s\right)=-\frac{1}{2}Ric_{\tilde{u}_s}Z_{h}\left(s\right)+h^{\prime}_s\text{ with }Z_{h}\left(0\right)=0.\label{equ.2.1}
\end{equation}
We will use $Z_\alpha$ as the shorthand of $Z_h$ when $h_s=\int_{0}^{s}\left(\tilde{T}_r^{-1}\right)^{\ast}e_{\alpha}dr$, $1\leq \alpha\leq d$.
\end{notation}
\begin{lemma}\label{lem Z}
Let $X^{Z_\alpha}$ be given above, then $Z_\alpha\in \mathcal{V}^{\infty}\cap \mathcal{B}^\infty$.
\end{lemma}
\begin{proof} Notice that $Z_{\alpha}$ satisfies the following ODE: 
\begin{equation}
Z_{\alpha}^{\prime}\left(s\right)=-\frac{1}{2}Ric_{\tilde{u}_s}Z_{\alpha}\left(s\right)+\left(\tilde{T}_s^{-1}\right)^{\ast}e_{\alpha}\text{ with }Z_{\alpha}\left(0\right)=0.
\end{equation}
Since $\left(\tilde{T}_s^{-1}\right)^{\ast}e_{\alpha}$ is adapted, $Z_{\alpha}^{\prime}$ is adapted. So $Z_{\alpha}$ is
a Brownian semimartingale with $Q\equiv0$ and $r=Z_{\alpha}^{\prime}$. Since $\tilde{T}_s$ is bounded, from Gronwall's inequality we have $Z_{\alpha}$ is bounded $\nu-a.s$, and the bound is independent of $\sigma\in W_o\left(M\right)$ and $s\in \left[0,1\right]$. Therefore $Z_\alpha\in \mathcal{V}^{\infty}\cap \mathcal{B}^\infty$.
  \end{proof}
\begin{definition}\label{equ.4.10-1}
	Define the domain of $\tilde{X}$ to be
	\[\mathcal{D}(\tilde{X}):=\cap_{\alpha=1}^d\mathcal{D}(X^{Z_\alpha})\]
	and for any $f\in \mathcal{D}(\tilde{X})$, set \[\tilde{X}f:=\sum_{\text{\ensuremath{\alpha=1}}}^{d}\left\langle \tilde{C}\tilde{H},e_{\alpha}\right\rangle X^{Z_\alpha}f,\] where $\tilde{C}=\left[\int_{0}^{1}\left(\tilde{T}_r^{\ast}\tilde{T}_r\right)^{-1}dr\right]^{-1}\tilde{T}_1^{-1}$.
\end{definition}
\begin{remark}
	To motivate this definition, we formally use the $H$-derivative. Notice that from Definition \ref{def Oc}:
\[\Phi_s=\int_{0}^{s}\left(\tilde{T}_{\tau}^{-1}\right){}^{\ast}\left[\int_{0}^{1}\left(\tilde{T}^{\ast}_r\tilde{T}_r\right){}^{-1}dr\right]^{-1}\tilde{T}_1^{-1}\tilde{H}d\tau=\sum_{\alpha=1}^{d}\left\langle \tilde{C}\tilde{H},e_{\alpha}\right\rangle \int_{0}^{s}\left(\tilde{T}_r^{-1}\right)^{\ast}e_{\alpha}dr,\]
by superposition principle, 
\[
Z_{\Phi}\left(s\right)=\sum_{\text{\ensuremath{\alpha=1}}}^{d}\left\langle \tilde{C}\tilde{H},e_{\alpha}\right\rangle Z_{\alpha}\left(s\right)
\]
and further 
\begin{equation}
X^{Z_{\Phi}}f=\left<Df,X^{Z_{\Phi}}\right>_{G^1}=\sum_{\text{\ensuremath{\alpha=1}}}^{d}\left\langle \tilde{C}\tilde{H},e_{\alpha}\right\rangle\left<Df,X^{Z_{\alpha}}\right>_{G^1}=\sum_{\text{\ensuremath{\alpha=1}}}^{d}\left\langle \tilde{C}\tilde{H},e_{\alpha}\right\rangle X^{Z_{\alpha}}f.\label{eq:-5}
\end{equation}
\end{remark}
\subsection{Computing $\tilde{X}^{tr,\nu}$ \label{sec.4.4}}
This subsection is devoted to the study of $\tilde{X}^{tr,\nu}$ (The
adjoint operator of $\tilde{X}$ with respect to $\nu$ restricted to $\mathcal{D}\left(\tilde{X}\right)$). The crucial step to show its existence is checking the anticipating coefficients in  $\left(\ref{eq:-5}\right)$ are differentiable in the sense of Definition  \ref{pro112}.

\begin{proposition}\label{prop Ric}
Our standard assumption of bounded curvature tensor implies that $Ric$ is bounded. If we further assume $\nabla R$ is bounded, then for any $h\in \mathcal{V}^{\infty-}\cap \mathcal{B}^{\infty-}$ and $s\in \left[0,1\right]$, we have $Ric_{\tilde{u}_s} \in \mathcal{D}\left(X^h\right)$. Moreover, the map
\begin{equation}
s\to X^hRic_{\tilde{u}_s}\in S^{\infty-}.\label{dif Ric}
\end{equation}
\end{proposition}
\begin{proof}
Since for any $s\in \left[0,1\right]$, $Ric_{\tilde{u}_s}\in \mathcal{GFC}^1$, from Theorem \ref{the.2.9} we know $Ric_{\tilde{u}_s}\in \mathcal{D}\left(X^h\right)$ and
\begin{equation}
X^hRic_{\tilde{u}_s}=\left(\nabla_{X_s^h}Ric\right)_{\tilde{u}_s}+\left[A_s\left<h\right>,Ric_{\tilde{u}_s}\right],\label{Ric1}
\end{equation} 
where $\left[\cdot,\cdot\right]$ is the Lie bracket of matrices and $\left(\nabla_{X_s^h}Ric\right)_{\tilde{u}_s}:\mathbb{R}^d\to \mathbb{R}^d$ is defined to be 
\[\left(\nabla_{X_s^h}Ric\right)_{\tilde{u}_s}={\tilde{u}_s}^{-1}(\nabla_{X_s^h}Ric)\cdot\tilde{u}_s.\]
Since $\nabla Ric$ is bounded, 
\begin{equation*}
\left|\left(\nabla_{X_s^h}Ric\right)_{\tilde{u}_s}\right|\leq C\left<X^h_s,X^h_s\right>_g^{\frac{1}{2}}= C\left\vert h_s\right\vert\leq Ch^*, 
\end{equation*}
where $C$ is a constant and $h^*$ is the essential supremum of $s\to h_s$. For any $q\in [1,\infty)$,  since $h\in \mathcal{B}^{\infty-}\subset S^{\infty-}$, we know
\begin{equation}
\sup_{s\in \left[0,1\right]}\left|\left(\nabla_{X_s^h}Ric\right)_{\tilde{u}_s}\right|\in L^{\infty-}\left(W_o\left(M\right)\right)\label{cRic}.
\end{equation}
Then we express $A_s\left<h\right>$ in It\^{o} form:
\[A_s\left<h\right>=\int_{0}^{s}R_{\tilde{u}_r}\left(h_r,d\beta_r\right)+\frac{1}{2}\sum_{i=1}^{d}\int_{0}^{s}\left\{R_{\tilde{u}_r}\left(Q^h_re_i,e_i\right)+\left(\frac{d}{dt}\mid_0R_{e^{tB_{e_i}}(\tilde{u}_r)}\right)\left(h_r,e_i\right)\right\}dr.\]
Since $R$ and $\nabla R$ are bounded, for any $s\in [0,1], q\geq 1$,
\begin{equation}
\left\vert \frac{1}{2}\sum_{i=1}^{d}\int_{0}^{s}\left\{R_{\tilde{u}_r}\left(Q^h_re_i,e_i\right)+\left(\frac{d}{dt}\mid_0R_{e^{tB_{e_i}}(\tilde{u}_r)}\right)\left(h_r,e_i\right)\right\}dr\right\vert^q\leq C_q\left(\left\Vert  Q^h\right\Vert^q_{L^q([0,1])}+(h^*)^q\right).\label{13}
\end{equation}
Using Burkholder-Davis-Gundy inequality, for any $q\in \left[1,\infty\right)$, 
\begin{equation}
\mathbb{E}\left[\sup_{s\in \left[0,1\right]}\left|\int_{0}^{s}R_{\tilde{u}_r}\left(h_r,d\beta_r\right)\right|^q\right]\leq C\left\Vert h\right\Vert^{\frac{q}{2}}_{L^{\frac{q}{2}}\left(W_o\left(M\right)\right)}<\infty.\label{14}
\end{equation}
Combining Eq.(\ref{13}) and (\ref{14}) we have
\[\sup_{s\in \left[0,1\right]}\left|A_s\left<h\right>\right|\in L^{\infty-}\left(W_o\left(M\right)\right).\]
Since $Ric$ is bounded, we have
\begin{equation}
\sup_{s\in \left[0,1\right]}\left\vert\left[A_s\left<h\right>,Ric_{\tilde{u}_s}\right]\right\vert\in L^{\infty-}\left(W_o\left(M\right)\right).\label{Ric}
\end{equation}
Combining (\ref{Ric1}), (\ref{cRic}) and (\ref{Ric}) gives (\ref{dif Ric}).
\end{proof}
\begin{lemma}\label{lmC}Let $\tilde{C}$ be as defined in Lemma \ref{equ.4.10-1}, then $\tilde{C}\in L^{\infty-}\left(W_o(M),\nu\right)$.
\end{lemma}
\begin{proof}
	Since $\left\Vert\tilde{T}_1^{-1}\right\Vert$ is bounded $\nu-a.s$, it suffices to show $	\left\Vert \left(\int_{0}^{1}\tilde{T}_r^{-1}(T_r^{-1})^{*}dr\right)^{-1}\right\Vert$ is bounded $\nu-a.s.$
	For any $v\in\mathbb{C}^{d},$ 
	\[
	\left\langle \left(\int_{0}^{1}\tilde{T}_r^{-1}(\tilde{T}_r^{-1})^{*}dr\right)v,v\right\rangle =\int_{0}^{1}\left\Vert (\tilde{T}_r^{-1})^{*}v\right\Vert ^{2}dr\geq C\left\Vert v\right\Vert ^{2}.\text{ }\nu-a.s.
	\]
	So 
	\[
	\left\Vert \left(\int_{0}^{1}\tilde{T}_r^{-1}(T_r^{-1})^{*}dr\right)^{-1}\right\Vert \leq\frac{1}{C}\text{ }\nu-a.s.
	\]
	where $C$ is a deterministic constant
\end{proof}
\begin{theorem}\label{thm419}Let $\tilde{T}_s$ be as defined in Definition \ref{def.T}, then $\tilde{T}_s\in \mathcal{D}\left(X^{Z_\alpha}\right)\text{  for  }1\leq \alpha \leq d.$
\end{theorem}
First we state a supplementary lemma.
\begin{lemma}\label{lem s}
	Let $\left\{f_{(\cdot)}(t)\right\}_{t\in \mathbb{R}}$ be $V-$valued Brownian semi-martingales which are $\mathcal{B}^p-$differentiable for some $p\geq 1$ at $t=0$, then for any $s\in [0,1]$, $\left\{f_{s}(t)\right\}_{t\in \mathbb{R}}$ are differentiable at $t=0$ in $L^p(W_o(M)\to V)$. Furthermore, 
	\[\left\Vert\frac{f_s(t)-f_s(0)}{t}-\frac{d}{dt}\mid_0f_{s}(t)\right\Vert_{L^p(W_o(M)\to V)}\to 0\text{ as }t\to 0\text{ uniformly with respect to s}.\]
\end{lemma}
\begin{proof}
	We represent $f_{(\cdot)}(t):=\int_{0}^{\cdot}Q_s^f(t)d\beta_s+\int_{0}^{\cdot}r_s^f(t)ds$ and denote $\frac{d}{dt}\mid_0f_{(\cdot)}(t)$ by \[g_{(\cdot)}:=\int_{0}^{\cdot}Q_s^gd\beta_s+\int_{0}^{\cdot}r_s^gds.\] $\frac{d}{dt}\mid_0f_{(\cdot)}(t)=g_{(\cdot)}$ in $\mathcal{B}^p$ implies that $\frac{d}{dt}\mid_0Q^f(t)=Q^g$ in $S^p(Hom(\mathbb{R}^d,V))$ and $\frac{d}{dt}\mid_0r^f(t)=r^g$ in $S^p(V)$.
	Since for any $s\in [0,1]$,  
	\begin{align*}
	&\left\Vert\frac{f_s(t)-f_s(0)}{t}-g_s\right\Vert^p_V\\&\leq C_p\left[\left\Vert \int_{0}^{s}\left(\frac{Q_\tau^f(t)-Q_\tau^f(0)}{t}-Q_\tau^g\right)d\beta_\tau \right\Vert^p_V+\left\Vert \int_{0}^{s}\left(\frac{r_\tau^f(t)-r_\tau^f(0)}{t}-r_\tau^g\right)d\tau  \right\Vert^p_V\right]
	\end{align*}
	Taking expectation on both hand side and using Burkholder-Davis-Gundy inequality on the first term, we have
	\[\left\Vert\frac{f_s(t)-f_s(0)}{t}-g_s\right\Vert_{L^p(W_o(M)\to V)}\leq C_p\left[\left\Vert \frac{Q_\tau^f(t)-Q_\tau^f(0)}{t}-Q_\tau^g \right\Vert_{\mathcal{S}^p}+\left\Vert \frac{r_\tau^f(t)-r_\tau^f(0)}{t}-r_\tau^g  \right\Vert_{S^p}\right]\to 0 \]
	as $t\to 0$. The uniformity with respect to $s$ is easily seen from the fact that the dominating function is independent of $s$.
\end{proof}

\begin{proof}[Proof of Theorem \ref{thm419}]
For each $X^{Z_\alpha}$, since $Z_\alpha\in \mathcal{V}^{\infty}\cap \mathcal{B}^{\infty}$ by Lemma \ref{lem Z},  we can construct an approximate flow $E\left(tX^{Z_\alpha}\right)$ of $X^{Z_\alpha}$. Define $\tilde{T}_s\left(t\right):=\tilde{T}_s\circ E\left(tX^{Z_\alpha}\right)$ and $G_s\left(t\right):=\frac{\tilde{T}_s\left(t\right)-\tilde{T}_s}{t}$, it is easy to see that $G_s\left(t\right)$ satisfies the following ODE:
\begin{equation*}
G_s^{\prime}\left(t\right)=-\frac{1}{2}Ric_{\tilde{u}_s}G_s\left(t\right)-\frac{1}{2t}\left(Ric_{\tilde{u}_s\left(t\right)}-Ric_{\tilde{u}_s}\right)\tilde{T}_s(t)\text{ with }G_0\left(t\right)=0,
\end{equation*}
where $\tilde{u}_{(\cdot)}(t)$ is the stochastic parallel translation along $E\left(tX^{Z_\alpha}\right)$ and $"\prime"$ is the derivative with respect to parameter $s$.

Then denote by $G_s$ the solution to the following ODE
\begin{equation*}
G_s^{\prime}=-\frac{1}{2}Ric_{\tilde{u}_s}G_s-\frac{1}{2}\left(X^{Z_\alpha}Ric_{\tilde{u}_s}\right)\tilde{T}_s\text{ with }G_0=0
\end{equation*} 
and let $H_s\left(t\right)$ be $G_s\left(t\right)-G_s$. We know $H_s\left(t\right)$ satisfies
\begin{equation*}
H_s^{\prime}\left(t\right)=-\frac{1}{2}Ric_{\tilde{u}_s}H_s\left(t\right)-\frac{1}{2}\left(\frac{Ric_{\tilde{u}_s\left(t\right)}-Ric_{\tilde{u}_s}}{t}\tilde{T}_s\left(t\right)+\left(X^{Z_\alpha}Ric_{\tilde{u}_s}\right)\tilde{T}_s\right)\text{, }H_0\left(t\right)=0.
\end{equation*}
According to Definition \ref{pro112}, \[\tilde{T}_s\in \mathcal{D}\left(X ^{Z_\alpha}\right)\iff H_s\left(t\right)\to 0 \text{ in }L^{\infty-}\left(W_o\left(M\right)\right).\]
By Gronwall's inequality, we have
\begin{equation*}
\left|H_s\left(t\right)\right|\leq \int_{0}^{s}\left|\frac{Ric_{\tilde{u}_r\left(t\right)}-Ric_{\tilde{u}_r}}{t}\tilde{T}_r\left(t\right)+\left(X^{Z_\alpha}Ric_{\tilde{u}_r}\right)\tilde{T}_r\right|dre^\frac{d\left(N-1\right)}{2}.
\end{equation*}
Following Theorem \ref{the.2.9} we know for any $p\geq 1$, $r\in [0,1]$, 
\begin{equation}
\frac{Ric_{\tilde{u}_r\left(t\right)}-Ric_{\tilde{u}_r}}{t}\to X^{Z_\alpha}Ric_{\tilde{u}_r}\text{ as }t\to 0 \text{ in }L^{p}(W_o(M)).\label{f1}
\end{equation}
 Since $Ric$, $\nabla Ric$ are bounded and $Z_\alpha\in \mathcal{V}^\infty\cap \mathcal{B}^\infty$, Lemma \ref{lem s} shows that this convergence is uniform with respect to $r\in \left[0,1\right]$. Since $\sup_{0\leq r\leq 1}\left\Vert\tilde{T}_r\right\Vert$ is bounded, using bounded convergence theorem, we have
 \begin{equation}
 \tilde{T}_r\left(t\right)\to \tilde{T}_r\text{ in }L^{\infty-}\left(W_o\left(M\right)\right)\text{ uniformly with respect to }r\in [0,1].\label{f2}
 \end{equation}
 Combining (\ref{f1}) and (\ref{f2}) we have $H_s\left(t\right)\to 0$ in $L^{\infty-}\left(W_o\left(M\right)\right)$ as $t\to 0$. 
\end{proof}
\begin{corollary}\label{Col C}
Recall that we have defined $\tilde{C}=\left[\int_{0}^{1}\left(\tilde{T}_r^{\ast}\tilde{T}_r\right)^{-1}dr\right]^{-1}\tilde{T}_1^{-1}$ in Lemma \ref{equ.4.10-1}, then
\[\tilde{C}\in \mathcal{D}\left(X^{Z_\alpha}\right)\text{ for }1\leq \alpha\leq d.\]
\end{corollary}
\begin{proof}
	Lemma \ref{lmC} shows that $\tilde{C}\in L^{\infty-}\left(W_o(M)\right)$. 
By the product rule and Theorem \ref{thm419}, for any $s\in \left[0,1\right]$, \[X^{Z_\alpha}\left(\tilde{T}_s^{-1}\right)=-\tilde{T}_s\left(X^{Z_\alpha}\tilde{T}_s\right)\tilde{T}_s\in L^{\infty-}\left(W_o\left(M\right)\right),\]so $\tilde{T}_s^{-1}\in \mathcal{D}\left(X^{Z_\alpha}\right)$ and thus $\int_{0}^{1}\left(\tilde{T}_r^{\ast}\tilde{T}_r\right)^{-1}dr\in \mathcal{D}\left(X^{Z_\alpha}\right)$. Then apply the product rule again we get $\tilde{C}\in \mathcal{D}\left(X^{Z_{\alpha}}\right)$.
\end{proof}
\begin{lemma} \label{lem.5.1}Given $X\in\Gamma\left(TM\right)$ with compact support, if $\tilde{X}$ is its orthogonal lift on $W_o\left(M\right)$, then define an operator on $L^2\left(W_o(M),\nu\right)$ by
\begin{align*}
\tilde{X}^{\operatorname{tr},\nu}= & -\tilde{X}+\sum_{\alpha=1}^{d}\left\langle \tilde{C}\tilde{H},e_{\alpha}\right\rangle \int_{0}^{1}\left\langle \left(\tilde{T}_s^{-1}\right)^{\ast}e_{\alpha},d\beta_{s}\right\rangle+\sum_{\alpha=1}^{d}\left\langle -X^{Z_{\alpha}}\left(\tilde{C}\tilde{H}\right),e_{\alpha}\right\rangle
\end{align*}
with $\mathcal{D}(\tilde{X}^{\operatorname{tr},\nu}):=\mathcal{D}(\tilde{X})$, then for any $f,g\in \mathcal{D}(\tilde{X})$, we have 
\[
\mathbb{E}_\nu\left[\tilde{X}f\cdot g\right]=\mathbb{E}_\nu\left[f\cdot\tilde{X}^{tr,\nu}g\right].
\] \end{lemma}

\begin{proof} 
Since $\tilde{H}\in \mathcal{GFC}^{1}$, $\tilde{H}\in \mathcal{D}\left(X^{Z_\alpha}\right)\text{  }\forall 1\leq \alpha\leq d$. Based on this observation and Corollary \ref{Col C}, we obtain
\begin{align}
\mathbb{E}\left[\tilde{X}f\cdot g\right]& =\mathbb{E}\left[\sum_{\text{\ensuremath{\alpha=1}}}^{d}\left\langle \tilde{C}\tilde{H},e_{\alpha}\right\rangle X^{Z_{\alpha}}f\cdot g\right]=\sum_{\alpha=1}^{d}\mathbb{E}\left[X^{Z_{\alpha}}f\cdot\left(g\cdot\left\langle \tilde{C}\tilde{H},e_{\alpha}\right\rangle \right)\right]=I+II+III\label{Eq:-10}
\end{align}
where 
\begin{align*}
& I=\mathbb{E}\left[f\cdot\left(-\tilde{X}\right)g\right]\\
& II=\mathbb{E}\left[f\cdot g\cdot\sum_{\alpha=1}^{d}\left\langle \tilde{C}\tilde{H},e_{\alpha}\right\rangle \int_{0}^{1}\left\langle \left(\tilde{T}_s^{-1}\right)^{\ast}e_{\alpha},d\beta_{s}\right\rangle \right]\\
& III=\mathbb{E}\left[f\cdot g\cdot\sum_{\alpha=1}^{d}\left\langle -X^{Z_{\alpha}}\left(\tilde{C}\tilde{H}\right),e_{\alpha}\right\rangle \right].
\end{align*}
Since $f\in L^{\infty-}\left(W_o(M),\nu\right)$, the proof can be completed by showing $\tilde{X}^{tr,\nu}g\in L^{\infty-}\left(W_o(M),\nu\right)$.
Corollary \ref{Col C} and the fact that $\tilde{H}\in \mathcal{D}(X^{Z_\alpha})$ implies that $\tilde{C}\tilde{H}, X^{Z_\alpha}\left(\tilde{C}\tilde{H}\right)\in
L^{\infty-}\left(W_o(M),\nu\right)$, so it suffices to show $\int_{0}^{1}\left\langle \left(\tilde{T}_s^{-1}\right)^{\ast}e_{\alpha},d\beta_{s}\right\rangle\in L^{\infty-}\left(W_o(M),\nu\right)$. The fact that it is true is a result of the boundedness of  $\underset{0\leq s\leq 1}{\sup}\left\Vert\tilde{T}_s^{-1}\right\Vert$  and Burkholder-Davis-Gundy inequality. 

\end{proof}

The following lemma gives a more
explicit expression of the last term in $\tilde{X}^{tr,\nu}$
\[
\sum_{\alpha=1}^{d}\left\langle -X^{Z_{\alpha}}\left(\tilde{C}\tilde{H}\right),e_{\alpha}\right\rangle 
\]
under an extra condition that $\nabla R\equiv 0$. The new expression indicates a structure of the divergence term $\tilde{X}^{tr,\nu}$ that is analogous to finite dimensional Riemannian geometry. Interested reader may refer to the structure theory on Appendix \ref{app.A}.
\begin{lemma} \label{lem.5.2}If further curvature tensor is parallel, i.e. $\nabla R\equiv 0$, then
\begin{align}
- & \sum_{\alpha=1}^{d}\left\langle X^{Z_{\alpha}}\left(\tilde{C}\tilde{H}\right),e_{\alpha}\right\rangle =divX\circ\Ep-\sum_{\alpha=1}^{d}\left\langle \tilde{C}A_1\left\langle Z_{\alpha}\right\rangle \tilde{H},e_{\alpha}\right\rangle .\label{equ.5.1}
\end{align}

\end{lemma}

\begin{proof} Since for tensors, contraction commutes with covariant differentiation, and $Ric$ is the contraction of curvature tensor $R$, so $\nabla Ric\equiv 0$ and thus $\delta Ric_{\tilde{u}_s}=\nabla_{\delta\beta_s}Ric\equiv 0$. So $Ric_{\tilde{u}_s}=Ric_{\tilde{u}_0}$ a.s. and it follows that $\tilde{T}_s$ and $\tilde{C}$ have deterministic versions.
	
	Since $\tilde{H}=\tilde{u}_1^{-1}X\left(\pi\circ \tilde{u}_1\right)\in\mathcal{GFC}^{1}$, we can apply Theorem \ref{the.2.9} to $\tilde{H}$ to find
\begin{align*}
\sum_{\alpha=1}^{d}\left\langle X^{Z_{\alpha}}\left(\tilde{C}\tilde{H}\right),e_{\alpha}\right\rangle  & =\sum_{\alpha=1}^{d}\left\langle \tilde{C}X^{Z_{\alpha}}\tilde{H},e_{\alpha}\right\rangle=I+II
\end{align*}
where 
\[
I=-\sum_{\alpha=1}^{d}\left\langle \tilde{C}\tilde{u}_1^{-1}\nabla_{X^{Z_{\alpha}}\left(1\right)}X,e_{\alpha}\right\rangle\text{ and  }\text{  } II=\sum_{\alpha=1}^{d}\left\langle \tilde{C}A_1\left\langle Z_{\alpha}\right\rangle \tilde{H},e_{\alpha}\right\rangle .
\]
\textbf{Claim: }$I=-divX\circ E_{1}.$
\[\]
\textbf{Proof of Claim:} 
\begin{align*}
I & =-\sum_{\alpha=1}^{d}\left\langle \tilde{u}_1\tilde{C}\tilde{u}_1^{-1}\nabla_{\tilde{u}_1\tilde{C}^{-1}\tilde{u}_1^{-1}\tilde{u}_1e_{\alpha}}X,\tilde{u}_1e_{\alpha}\right\rangle=-\sum_{\alpha=1}^{d}\left\langle A^{-1}\nabla_{Af_{\alpha}}X,f_{\alpha}\right\rangle =-\sum_{\alpha=1}^{d}\left\langle \nabla_{Af_{\alpha}}X,\left(A^{-1}\right)^{\ast}f_{\alpha}\right\rangle 
\end{align*}
where $A=\tilde{u}_1\tilde{C}^{-1}\tilde{u}_1^{-1}\in End\left(T_{\Ep\left(\sigma\right)}M\right)$
and $\left\{ f_{\alpha}\right\} =\left\{ \tilde{u}_1e_{\alpha}\right\} $
is an orthonormal basis of $T_{\Ep\left(\sigma\right)}M$. Since $\left\langle \nabla_{\cdot}X,\cdot\right\rangle $ is
bilinear on $T_{\Ep\left(\sigma\right)}M$, by the universal property
of tensor product we know there exists a linear map $l:T_{\Ep\left(\sigma\right)}M\otimes T_{\Ep\left(\sigma\right)}M\mapsto\mathbb{R}$
such that 
\[
\left\langle \nabla_{Af_{\alpha}}X,\left(A^{-1}\right)^{\ast}f_{\alpha}\right\rangle =l\left(Af_{\alpha}\otimes\left(A^{-1}\right)^{\ast}f_{\alpha}\right)
\]
and therefore: 
\begin{equation}
\sum_{\alpha=1}^{d}\left\langle \nabla_{Af_{\alpha}}X,\left(A^{-1}\right)^{\ast}f_{\alpha}\right\rangle =l\left(\sum_{\alpha=1}^{d}Af_{\alpha}\otimes\left(A^{-1}\right)^{\ast}f_{\alpha}\right).\label{equ.5.2}
\end{equation}
Using the isomorphism between $T^{1,1}\left(V\right)\mapsto End\left(V\right):$$\left(a\otimes b\right)v=a\cdot\left\langle b,v\right\rangle $
one can easily see: 
\begin{equation}
\sum_{\alpha=1}^{d}Af_{\alpha}\otimes\left(A^{-1}\right)^{\ast}f_{\alpha}=\sum_{\alpha=1}^{d}f_{\alpha}\otimes f_{\alpha}.\label{equ.5.3}
\end{equation}
Combining $\left(\ref{equ.5.2} \right)$ and $\left(\ref{equ.5.3}\right)$ we have
\[
I=-\sum_{\alpha=1}^{d}\left\langle \nabla_{f_{\alpha}}X,f_{\alpha}\right\rangle =-divX\circ \Ep
\]
and thus $\left(\ref{equ.5.1}\right)$. \end{proof}

\appendix

\section{ODE estimates}\label{App.B}
\begin{lemma}\label{Lem1}
	Let $\alpha\left(t\right)\in\operatorname*{End}\left(\mathbb{R}^{d}\right)$ be a continuously varying matrix valued function and $S(t)\in \operatorname*{End}\left(\mathbb{R}^{d}\right)$ be the solution to the following initial value problem:
	\[\frac{d}{dt}S(t)=\alpha(t)S(t),\text{ }S(0)=I,\]
	then for any $t\in [0,1]$, $S(t)\in \operatorname*{Aut}\left(\mathbb{R}^{d}\right)$. Furthermore,
	\[\int_{0}^{t}[S(r)^*S(r)]^{-1}dr\in \operatorname*{Aut}\left(\mathbb{R}^{d}\right)\text{ }\forall t\in [0,1].\]
\end{lemma}
\begin{proof}
	Denote by $U(t)\in \operatorname*{End}\left(\mathbb{R}^{d}\right)$ the solution to the following initial value problem:
	\[\frac{d}{dt}U(t)=-U(t)\alpha(t),\text{ }U(0)=I,\]
	then direct computation shows that $Y(t):=S(t)U(t)\in \operatorname*{End}\left(\mathbb{R}^{d}\right)$ satisfies 
	\[\frac{d}{dt}Y(t)=\alpha(t)Y(t)-Y(t)\alpha(t),\text{ }Y(0)=I.\]
	By the uniqueness of solutions for linear ODE, we get $S(t)U(t)\equiv I$,  
	and this shows that $U{(t)}$ is a left inverse to $S{(t)}$.
	As we are in finite dimensions it follows that $T{(t)}^{-1}$
	exists and is equal to $U{(t)}.$
	Then for any $v\in \mathbb{C}^d/\left\{0\right\}$, 
	\begin{align}
	\left<\int_{0}^{t}[S(r)^*S(r)]^{-1}drv,v\right>&=\int_{0}^{t}\left<[S(r)^*]^{-1}v,[S(r)^*]^{-1}v\right>dr=\int_{0}^{t}\left\Vert U(r)v\right\Vert^2dr
	\end{align} 
	Since $U(0)=I$ and $U(\cdot):[0,1]\to \operatorname*{Aut}\left(\mathbb{R}^{d}\right)$ is continuous, 
	\[\left<\int_{0}^{t}[S(r)^*S(r)]^{-1}drv,v\right>>0\]
	and this implies $\int_{0}^{t}[S(r)^*S(r)]^{-1}dr\in \operatorname*{Aut}\left(\mathbb{R}^{d}\right)\forall t\in [0,1].$
\end{proof}
\section{A Structure Theorem for $\operatorname*{div}\nolimits _{g}\left(\tilde{X}\right)$\label{app.A}}

This section is devoted to a structure theorem for  $\operatorname*{div}\nolimits _{g}\left(\tilde{X}\right)$---the divergence of the lifted vector field $\tilde{X}$ in finite dimensional Riemannian geometry. We expect that the orthogonal lift that we introduced in this paper also has an analogous structure, as is hinted in Lemma \ref{lem.5.2}.

Let $\pi:\left(M,g\right)\rightarrow\left(N,h\right)$ be a submersion
of two smooth Riemannian manifolds. To each $m\in M$ and $v\in T_{\pi\left(m\right)}N,$
let $\hat{v}:=\pi_{\ast m}^{\operatorname{tr}}\left(\pi_{\ast m}\pi_{\ast m}^{\operatorname{tr}}\right)^{-1}v\in T_{m}M$
so that $\hat{v}$ is the unique shortest vector in $T_{m}M$ such
that $\pi_{\ast m}\hat{v}=v.$ So if $X\in\Gamma\left(TN\right)$
is a vector field on $N,$ then $\hat{X}\in\Gamma\left(TM\right)$
is defined by $\hat{X}\left(m\right)=\pi_{\ast m}^{\operatorname{tr}}\left(\pi_{\ast m}\pi_{\ast m}^{\operatorname{tr}}\right)^{-1}X\left(\pi\left(m\right)\right)$
and we have $\pi_{\ast}\hat{X}=X\circ\pi.$ Finally, let $\operatorname*{Vol}_{g}$
and $\operatorname*{Vol}_{h}$ be the volume forms on $\left(M,g\right)$
and $\left(N,h\right)$ respectively.

\begin{lemma} \label{lem.A.1}If $K:=\dim M>k:=\dim N,$ then there
	exists a unique $K-k$ -- form $\left(\gamma\right)$ on $M$ such
	that; 
	\begin{enumerate}
		\item $\operatorname*{Vol}_{g}=\left(\pi^{\ast}\operatorname*{Vol}_{h}\right)\wedge\gamma$ 
		\item $i_{\hat{v}}\gamma=0$ for any $v\in T_{\pi\left(m\right)}N$ and
		$m\in M.$ 
	\end{enumerate}
\end{lemma}

\begin{proof} \textbf{Uniqueness. }Assuming such a $\gamma$ exists,
	choose an orthonormal basis $\left\{ e_{1},\dots,e_{k}\right\} $
	for $T_{\pi\left(m\right)}N$ such that $\operatorname*{Vol}_{h}\left(e_{1},\dots,e_{k}\right)=1.$
	Then it follows that 
	\begin{align*}
	\operatorname*{Vol}\nolimits _{g}\left(\hat{e}_{1},\dots,\hat{e}_{k},\cdot,\dots,\cdot\right) & =\left(\pi^{\ast}\operatorname*{Vol}\nolimits _{h}\right)\left(\hat{e}_{1},\dots,\hat{e}_{k}\right)\wedge\gamma\\
	& =\operatorname*{Vol}\nolimits _{h}\left(\pi_{\ast}\hat{e}_{1},\dots,\pi_{\ast}\hat{e}_{k}\right)\wedge\gamma\\
	& =\operatorname*{Vol}\nolimits _{h}\left(e_{1},\dots,e_{k}\right)\wedge\gamma=\gamma
	\end{align*}
	which shows $\gamma$ is unique if it exists.
	
	\textbf{Existence. }Now suppose that $\left\{ e_{1},\dots,e_{k}\right\} $
	is a local orthonormal frame on $M$ in a neighborhood of $\pi\left(m\right)$
	such that $\operatorname*{Vol}_{h}\left(e_{1},\dots,e_{k}\right)=1.$
	Then by above we must define 
	\[
	\gamma:=\operatorname*{Vol}\nolimits _{g}\left(\hat{e}_{1},\dots,\hat{e}_{k},\cdot,\dots,\cdot\right)\text{ in a neighborhood of }m.
	\]
	It is now straightforward to check that this $\gamma$ has the desired
	properties and is defined independent of the choice of frame. \end{proof}

\begin{corollary} \label{cor.A.2}If $X\in\Gamma\left(TN\right)$
	and $\hat{X}\in\Gamma\left(TM\right)$ is its lift as described above,
	then 
	\[
	\operatorname*{div}\nolimits _{g}\left(\hat{X}\right)=\operatorname*{div}\nolimits _{h}\left(X\right)\circ\pi+\rho_{\hat{X}}
	\]
	where $\rho_{\hat{X}}\left(m\right)$ is a function on $M$ depending
	only on $\hat{X}\left(m\right).$ \{To compute $\rho_{\hat{X}}$ explicitly
	will require a better understanding of $d\gamma.]$ \end{corollary}

\begin{proof} From Lemma \ref{lem.A.1} we learn, 
	\begin{align*}
	\operatorname*{div}\nolimits _{g}\left(\hat{X}\right)\operatorname*{Vol}\nolimits _{g} & =d\left[i_{\hat{X}}\operatorname*{Vol}\nolimits _{g}\right]=d\left[i_{\hat{X}}\left(\left(\pi^{\ast}\operatorname*{Vol}\nolimits _{h}\right)\wedge\gamma\right)\right]\\
	& =d\left[\left(i_{\hat{X}}\left(\pi^{\ast}\operatorname*{Vol}\nolimits _{h}\right)\wedge\gamma\right)\right]\\
	& =\left[d\left(i_{\hat{X}}\left(\pi^{\ast}\operatorname*{Vol}\nolimits _{h}\right)\right)\right]\wedge\gamma+\left(-1\right)^{k}\left(i_{\hat{X}}\left(\pi^{\ast}\operatorname*{Vol}\nolimits _{h}\right)\wedge d\gamma\right).
	\end{align*}
	Since 
	\begin{align*}
	i_{\hat{X}}\left(\pi^{\ast}\operatorname*{Vol}\nolimits _{h}\right) & =\left(\pi^{\ast}\operatorname*{Vol}\nolimits _{h}\right)\left(\hat{X},--\right)=\operatorname*{Vol}\nolimits _{h}\left(\pi_{\ast}\hat{X},\pi_{\ast}--\right)\\
	& =\operatorname*{Vol}\nolimits _{h}\left(X\circ\pi,\pi_{\ast}--\right)=\pi^{\ast}\left(i_{X}\operatorname*{Vol}\nolimits _{h}\right)
	\end{align*}
	it follows that 
	\begin{align*}
	d\left(i_{\hat{X}}\left(\pi^{\ast}\operatorname*{Vol}\nolimits _{h}\right)\right) & =d\left(\pi^{\ast}\left(i_{X}\operatorname*{Vol}\nolimits _{h}\right)\right)=\pi^{\ast}\left(d\left(i_{X}\operatorname*{Vol}\nolimits _{h}\right)\right)\\
	& =\pi^{\ast}\left(\operatorname*{div}\nolimits _{h}\left(X\right)\operatorname*{Vol}\nolimits _{h}\right)=\operatorname*{div}\nolimits _{h}\left(X\right)\circ\pi\cdot\pi^{\ast}\operatorname*{Vol}\nolimits _{h}.
	\end{align*}
	Combining these equations then shows, 
	\begin{align*}
	\operatorname*{div}\nolimits _{g}\left(\hat{X}\right)\operatorname*{Vol}\nolimits _{g} & =\operatorname*{div}\nolimits _{h}\left(X\right)\circ\pi\cdot\left(\pi^{\ast}\operatorname*{Vol}\nolimits _{h}\right)\wedge\gamma+\left(-1\right)^{k}\left(i_{\hat{X}}\left(\pi^{\ast}\operatorname*{Vol}\nolimits _{h}\right)\wedge d\gamma\right)\\
	& =\left[\operatorname*{div}\nolimits _{h}\left(X\right)\circ\pi+\rho_{\hat{X}}\right]\cdot\operatorname*{Vol}\nolimits _{g}
	\end{align*}
	where 
	\[
	\rho_{\hat{X}}=\frac{\left(-1\right)^{k}\left(i_{\hat{X}}\left(\pi^{\ast}\operatorname*{Vol}\nolimits _{h}\right)\wedge d\gamma\right)}{\operatorname*{Vol}\nolimits _{g}}.
	\]

\end{proof}

\bibliographystyle{amsplain}
\bibliography{references}

\end{document}